\documentclass[12pt,reqno]{amsart}
\allowdisplaybreaks

\usepackage{amssymb}
\usepackage[nobysame]{amsrefs}
\usepackage{todonotes}
\usepackage{enumerate}
\usepackage{subcaption}
\usepackage{tikz,pgf,pgfmath}
	\usetikzlibrary{math,calc}
	\pgfdeclarelayer{background}
	\pgfsetlayers{background,main}
	\tikzstyle{edge}=[line width=.67pt]
\usepackage[colorlinks=true,
	urlcolor=blue!20!black,
	citecolor=green!50!black]{hyperref}
\newtheorem{theorem}{Theorem}[section]
\newtheorem{proposition}[theorem]{Proposition}
\newtheorem{conjecture}[theorem]{Conjecture}
\newtheorem{lemma}[theorem]{Lemma}
\newtheorem{corollary}[theorem]{Corollary}
\theoremstyle{remark}
\newtheorem{definition}[theorem]{Definition}
\newtheorem{example}[theorem]{Example}
\newtheorem{remark}[theorem]{Remark}

\numberwithin{equation}{section}

\newcommand{\defn}[1]{{\color{green!50!black}\emph{#1}}}
\newcommand{\defs}{\stackrel{\mathrm{def}}{=}}
\newcommand{\ie}{\text{i.e.},\;}
\newcommand{\eg}{\text{e.g.},\;}

\def\rk{\operatorname{rk}}
\def\block{\operatorname{bl}}
\def\NC{\operatorname{\bf NC}}
\def\NN{\operatorname{\bf NN}}
\def\fl{\operatorname{fl}}
\def\FL{\operatorname{FL}}
\def\valley{\operatorname{v}}
\def\peak{\operatorname{p}}
\def\return{\operatorname{r}}
\def\dref{\leq_{\mathrm{ref}}}
\def\dom{\leq_{\mathrm{ddom}}}
\def\ZetaPol{\mathcal{Z}}
\def\coef#1{\left\langle#1\right\rangle}
\def\Dyck{\mathcal{D}}
\title{The Rank Enumeration of Certain Parabolic Non-Crossing Partitions}
\author[C.~Krattenthaler]{Christian Krattenthaler}
\address[C.~Krattenthaler]{Universit{\"a}t Wien, Fakult{\"a}t f{\"u}r Mathematik, Oskar-Morgenstern-Platz 1, A-1090 Vienna, Austria. WWW: {\tt https://www.mat.univie.ac.at/\lower0.5ex\hbox{\~{}}kratt}.}
\email{christian.krattenthaler@univie.ac.at}
\author[H.~M\"uhle]{Henri M{\"u}hle}
\address[H.~M\"uhle]{Technische Universit{\"a}t Dresden, Institut f{\"u}r Algebra, Zellescher Weg 12--14, 01069 Dresden, Germany.
}
\email{henri.muehle@tu-dresden.de}
\begin{document}

\allowdisplaybreaks

\begin{abstract}
	We consider $m$-divisible non-crossing partitions of $\{1,2,\ldots,mn\}$ with the property that for some $t\leq n$ no block contains more than one of the integers $1,2,\ldots,t$.  We give a closed formula for the number of multi-chains of such non-crossing partitions with prescribed number of blocks.  Building on this result, we compute Chapoton's $M$-triangle in this setting and conjecture a combinatorial interpretation for the $H$-triangle.  This conjecture is proved for $m=1$.
\end{abstract}

\maketitle

\section{Introduction}

Non-crossing partitions have appeared in the combinatorial landscape in the early 1970s, and---despite their simple definition---have ever since served an important, connecting purpose for many different aspects of algebraic combinatorics.  Even though they were considered under different names before (\eg as planar rhyme schemes~\cite{becker52planar}), their systematic study has begun only after Kreweras' seminal article~\cite{kreweras72sur}.

A (set) partition is a covering of the set $[n]\defs\{1,2,\ldots,n\}$ by non-empty, mutually disjoint sets (so-called blocks) and it is non-crossing if there do not exist indices $1\leq i<j<k<l\leq n$ such that $i,k$ belong to one block and $j,l$ to another.  We obtain additional structure if we order non-crossing partitions by refinement, \ie one partition refines another if every block of the first partition is contained in some block of the second.  In fact, the resulting partial order (denoted by $\dref$) endows the set of non-crossing partitions with a rank function (where the rank is $n$ minus the number of blocks) and a lattice structure.  This lattice of non-crossing partitions was studied intensively from enumerative and structural points of view in \cites{edelman80chain,edelman82multichains,edelman94chains,simion91on}.  The notion of $m$-divisible non-crossing partitions goes back to Edelman~\cite{edelman80chain}, and refers to non-crossing partitions in which every block has size divisible by $m$.

A remarkable observation by Biane~\cite{biane97some} (and independently Brady~\cite{brady01partial}) relates the lattice of non-crossing partitions to a certain interval in the Cayley graph of the symmetric group with respect to the generating set of all transpositions.  This construction was generalised by Bessis~\cite{bessis03dual} and by Brady and Watt~\cite{brady08non} to finite irreducible Coxeter groups, for which combinatorial interpretations were given by Athanasiadis and Reiner~\cites{athanasiadis04noncrossing,reiner97non}.  Armstrong subsequently finished this stream of generalisation by detailing how $m$-divisible non-crossing partitions are constructed for finite irreducible Coxeter groups~\cite{armstrong09generalized}.

Most recently, Williams introduced a generalisation of non-crossing partitions to parabolic quotients of finite irreducible Coxeter groups~\cite{williams13cataland}, and this construction was studied later for parabolic quotients of the symmetric group in \cites{muehle19tamari,muehle21noncrossing}.  In this article, we study a certain family of these parabolic non-crossing partitions, namely those where one takes a parabolic quotient of the symmetric group with respect to an initial segment of the list of adjacent transpositions.  This particular family of parabolic non-crossing partitions admits the following combinatorial interpretation.  We fix positive integers $n,t$ such that $t\leq n$, and define a \defn{$t$-partition} to be a partition of $[n]$ in which no block intersects the set $[t]$ in more than one element.  Such a $t$-partition is non-crossing if there exist no four indices $1\leq i<j<k<l\leq n$ such that either $j\leq t$ and $i,l$ belong to one block and $j,k$ to another, or $t<j$ and $i,k$ belong to one block and $j,l$ to another.  For $t=1$ this construction reduces to the ordinary non-crossing partitions introduced in the first two paragraphs.

Our main result provides a closed formula for the number of multi-chains of $m$-divisible non-crossing $t$-partitions under refinement with prescribed ranks.  For convenience, we denote the set of all $m$-divisible non-crossing $t$-partitions of $[mn]$ by $\NC_{n;t}^{(m)}$.

\begin{theorem}\label{thm:nt_rank_enumeration}
	Let $m,n,t,l$ be positive integers, and let $s_{1},s_{2},\ldots,s_{l+1}$ be non-negative integers with $s_{1}+s_{2}+\cdots+s_{l+1}=n-t$.  The number of multi-chains $\pi_{1}\dref\pi_{2}\dref\cdots\dref\pi_{l}$, where $\pi_{i}\in\NC_{n;t}^{(m)}$ and $\rk(\pi_{i})=s_{1}+s_{2}+\cdots+s_{i}$ for $i\in[l]$, is given by
	\begin{equation}\label{eq:nt_rank_enumeration}
		\frac{t(mn-t+1)-s_{l+1}(t-1)}{n(mn-t+1)}\binom{n}{s_{1}}\binom{mn}{s_{2}}\cdots\binom{mn}{s_{l}}\binom{mn-t+1}{s_{l+1}}.
	\end{equation}
\end{theorem}

The proof of Theorem~\ref{thm:nt_rank_enumeration} utilises generating functions and is carried out in Section~\ref{sec:rank_enumeration}, after the necessary definitions have been recalled in Section~\ref{sec:preliminaries}.

As an application of Theorem~\ref{thm:nt_rank_enumeration} we compute a certain bivariate integer polynomial, which can be seen as the generating function of the intervals of non-crossing $m$-divisible $t$-partitions under refinement, weighted by the M{\"o}bius function.  We call this polynomial the \defn{$M$-triangle}, denoted by $M_{n;t}^{(m)}$.  In the case $t=1$, it was observed by Chapoton (for $m=1$) and Armstrong (for $m>1$) that under certain substitutions of the variables, one obtains two other bivariate integer polynomials, called the $F$- and $H$-triangle, respectively~\cites{armstrong09generalized,chapoton04enumerative,chapoton06sur}.  These references also provide concrete combinatorial realisations of those polynomials.

We show in Section~\ref{sec:m_triangle} that the same substitutions---when applied to $M_{n;t}^{(m)}$---yield integer polynomials, too, see Theorem~\ref{thm:fhm_correspondence}.  In the last part of this article, Section~\ref{sec:h_triangle}, we conjecture a combinatorial realisation of the $H$-triangle in terms of multi-chains of filters in a certain partially ordered set on transpositions, see Conjecture~\ref{conj:h_triangle}.  We conclude by proving this realisation in the case $m=1$ (again with the help of generating functions) using a lattice path model, see Theorem~\ref{thm:h_triangle_m=1}.

\section{Preliminaries}
	\label{sec:preliminaries}
\subsection{Non-crossing partitions}

Given an integer $n>0$, a \defn{(set) partition} is a family of non-empty, mutually disjoint sets, called \defn{blocks}, whose union is all of $[n]\defs\{1,2,\ldots,n\}$.  For $m>0$, a partition is \defn{$m$-divisible} if every block has size divisible by $m$.  Given a partition $\pi$, we write $\block(\pi)$ for the number of blocks of $\pi$, and define the \defn{rank} of $\pi$ by $\rk(\pi)\defs n-\block(\pi)$.  A partition $\pi$ \defn{refines} a partition $\pi'$ if every block of $\pi$ is contained in some block of $\pi'$; we denote this relation by $\pi\dref\pi'$.

A partition is \defn{non-crossing} if there are no indices $i<j<k<l$ such that $i,k$ belong to one block and $j,l$ belong to a different block.  
We denote the set of all $m$-divisible non-crossing partitions of $[mn]$ by $\NC_{n}^{(m)}$.  
Non-crossing partitions were introduced in \cite{kreweras72sur}, and the $m$-divisible case was studied in \cite{edelman80chain}.  An excellent survey on $m$-divisible non-crossing partitions and various related objects is \cite{armstrong09generalized}.

For $t\in[n]$, a \defn{$t$-partition} is a partition where no block intersects the set $[t]$ in more than one element.  A $t$-partition is \defn{non-crossing} if there do not exist indices $i<j<k<l$ such that either $j\leq t$ and $i,l$ belong to one block and $j,k$ belong to a different block, or $j>t$ and $i,k$ belong to one block and $j,l$ belong to another block.  We denote the set of all $m$-divisible non-crossing $t$-partitions of $[mn]$ by $\NC_{n;t}^{(m)}$.  Clearly, we have $\NC_{n;1}^{(m)}=\NC_{n}^{(m)}$.

The non-crossing condition can be visualised in terms of \defn{arc diagrams}.  Given a $t$-partition $\pi$ of $[n]$, we draw $n$ nodes, labelled by $1,2,\ldots,n$, on a horizontal line, where we colour the nodes in $[t]$ in white and the remaining nodes in black.  Then, we connect two nodes $i$ and $j$ by an arc if and only if $i$ and $j$ are in the same block of $\pi$ and every $k\in\{i{+}1,i{+}2,\ldots,j{-}1\}$ is in a different block.  This arc stays below the nodes $i{+}1,i{+}2,\ldots,t$ and rises above the nodes $\max\{i,t\}{+}1,\max\{i,t\}{+}2,\ldots,j$.  (If $t=1$, then we use the convention that all nodes are black.)  See Figure~\ref{fig:arc_diagrams}.  Then, an $m$-divisible $t$-partition of $[mn]$ is non-crossing if and only if its arc diagram can be drawn in such a way that no two arcs cross.

\begin{figure}
	\centering
	\begin{subfigure}[t]{.9\textwidth}
		\centering
		\begin{tikzpicture}\small
			\def\x{.5};
			\foreach \i in {1,...,3}{
				\draw(\i*\x,1) node[circle,draw,scale=.4,label=above:{\tiny \i}](w\i){};
			}
			\foreach \i in {4,...,14}{
				\draw(\i*\x,1) node[circle,draw,fill,scale=.4,label=below:{\tiny \i}](b\i){};
			}
			\draw[edge](w1) to[bend right=70] (3.7*\x,1) to[bend left=70] (b6) to[bend left=70] (b7) to[bend left=70] (b8);
			\draw[edge](w2) to[bend right=70] (3.5*\x,1) to[bend left=70] (b9) to[bend left=70] (b12) to[bend left=70] (b13);
			\draw[edge](w3) to[bend right=70] (3.3*\x,1) to[bend left=70] (b14);
			\draw[edge](b4) to[bend left=70] (b5);
			\draw[edge](b10) to[bend left=70] (b11);
		\end{tikzpicture}
		\caption{The $2$-divisible $3$-partition $\bigl\{\{1,6,7,8\},\{2,9,12,13\},\{3,14\},\{4,5\},\{10,11\}\bigr\}$ is non-crossing.}
		\label{fig:arc_diagram_noncrossing}
	\end{subfigure}

	\begin{subfigure}[t]{.9\textwidth}
		\centering
		\begin{tikzpicture}\small
			\def\x{.5};
			\foreach \i in {1,...,3}{
				\draw(\i*\x,1) node[circle,draw,scale=.4,label=above:{\tiny \i}](w\i){};
			}
			\foreach \i in {4,...,14}{
				\draw(\i*\x,1) node[circle,draw,fill,scale=.4,label=below:{\tiny \i}](b\i){};
			}
			\draw[edge](w1) to[bend right=70] (3.7*\x,1) to[bend left=70] (b6) to[bend left=70] (b7) to[bend left=70] (b8);
			\draw[edge](w2) to[bend right=70] (3.5*\x,1) to[bend left=70] (b9) to[bend left=70] (b12) to[bend left=70] (b13);
			\draw[edge](w3) to[bend right=70] (3.3*\x,1) to[bend left=70] (b4);
			\draw[edge](b4) to[bend left=70] (b5);
			\draw[edge](b5) to[bend left=70] (b14);
			\draw[edge](b10) to[bend left=70] (b11);
		\end{tikzpicture}
		\caption{The $2$-divisible $3$-partition $\bigl\{\{1,6,7,8\},\{2,9,12,13\},\{3,4,5,14\},\{10,11\}\bigr\}$ is crossing.}
		\label{fig:arc_diagram_crossing}
	\end{subfigure}
	\caption{Some arc diagrams of $2$-divisible $3$-partitions.}
	\label{fig:arc_diagrams}
\end{figure}

In this article, we want to study the poset\footnote{\defn{Poset} is short for partially ordered set.} of $m$-divisible non-crossing $t$-partitions under refinement.  The next lemma enables us to reuse many of the well-known properties of the case $t=1$.

\begin{lemma}\label{lem:t_nc_embeds_in_nc}
	The poset $\Bigl(\NC_{n;t}^{(m)},\dref\Bigr)$ of $m$-divisible non-crossing $t$-partitions under refinement is isomorphic to an order ideal\footnote{An \defn{order ideal} in a poset is a collection of poset elements which is closed under going down in the order.} in $\Bigl(\NC_{n}^{(m)},\dref\Bigr)$.
\end{lemma}
\begin{proof}
	If $\pi$ is a $t$-partition, then we define the transformed partition $\tilde{\pi}$ via the map 
	\begin{displaymath}
		i \mapsto \begin{cases}t+1-i, & i\leq t,\\i, & i>t.\end{cases}
	\end{displaymath}
	Essentially, $\tilde{\pi}$ is obtained by rotating all the white nodes by 180 degrees counterclockwise.  It is then clear that $\pi\in\NC_{n;t}^{(m)}$ if and only if $\tilde{\pi}\in\NC_{n;1}^{(m)}$, and the claim follows.
\end{proof}

\begin{corollary}\label{cor:t_nc_graded}
	The poset $\Bigl(\NC_{n;t}^{(m)},\dref\Bigr)$ is graded of rank $n-t$, and the rank of $\pi\in\NC_{n;t}^{(m)}$ is $n-\block(\pi)$.
\end{corollary}
\begin{proof}
	The claim that $\Bigl(\NC_{n;t}^{(m)},\dref\Bigr)$ is graded, and the claim that the rank of its elements is $n$ minus the number of blocks follows from Lemma~\ref{lem:t_nc_embeds_in_nc} and \cite{edelman80chain}*{Section~4}.  It remains to determine the maximal rank of elements of $\NC_{n;t}^{(m)}$, or equivalently the minimal number of blocks.  But this is clearly $t$, since by definition each element of $[t]$ must be in its own block, and this value is attained for instance by the partition 
	\begin{multline*}
		\bigl\{\{1,  t{+}1,t{+}2,\ldots,t{+}m{-}1\},\{2,t{+}m,t{+}m{+}1,\ldots,t{+}2m{-}2\},\ldots,\\
		 \{t{-}1,(t{-}2)m{+}3,(t{-}2)m{+}4,\ldots,(t{-}1)m{+}1\},\{t,(t{-}1)m{+}2,(t{-}1)m{+}3,\ldots,mn\}\bigr\}.
	\end{multline*}
\end{proof}

\begin{remark}
	For $m=1$, we recover a particular case of the \emph{parabolic non-crossing partitions} first considered in \cites{muehle19tamari,williams13cataland}, and later studied in \cite{muehle21noncrossing}.  
\end{remark}

\begin{remark}
	We may view the elements of $\NC_{n;t}^{(m)}$ as certain non-crossing partitions on an annulus.  To that end, we place $t$ nodes, labelled by $1,2,\ldots,t$, on the inner boundary of an annulus (in clockwise order), and $mn-t$ nodes, labelled by $t{+}1,t{+}2,\ldots,mn$, on the outer boundary (in clockwise order, too).  
	
	Given an $m$-divisible partition $\pi$, we inscribe the convex hulls of the blocks of $\pi$ on this annulus.  If none of these convex hulls touches the inner boundary more than once, then $\pi$ is in fact a $t$-partition.  Modelling the non-crossing condition, however, is a bit more intricate.  One requirement, certainly, is that no two convex hulls intersect in their interior.  The other requirement is the following: if $B$ is a non-singleton block containing the node $i$ for $i\in[t]$, let $j$ be the smallest node in $B\setminus\{i\}$.  This block casts a shadow on the nodes $t{+}1,t{+}2,\ldots,j{-}1$ which is impenetrable for blocks containing $i'$ for $i<i'\leq t$.  Indeed, if we suppose that there is a block $B'$ containing $i'$ and $j'\in\{t{+}1,t{+}2,\ldots,j{-}1\}$, then the blocks $B$ and $B'$ would violate the non-crossing condition for $t$-partitions.

	\newcounter{lab}
	
	\begin{figure}
		\centering
		\begin{subfigure}[t]{.4\textwidth}
			\centering
			\begin{tikzpicture}[scale=.8]\small
				\def\di{.75};
				\def\do{2.5};
				\foreach \i in {1,...,3} {
					\coordinate (p\i) at ($(-\i*120+20:\di)$);
					\coordinate (pl\i) at ($(p\i) - (-\i*120+20:\di/2.5)$);
					\filldraw[draw=black,fill=white] (p\i) circle(2pt);
					\draw (pl\i) node{\tiny\i};
				}
				\foreach \i in {1,...,11} {
					\coordinate (q\i) at ($(-\i*360/11+360/11:\do)$);
					\coordinate (ql\i) at ($(q\i) + (-\i*360/11+360/11:\do/6)$);
					\filldraw[black] (q\i) circle(2pt);
					\setcounter{lab}{3}
					\addtocounter{lab}{\i}
					\draw (ql\i) node{\tiny\thelab};
				}
				\begin{pgfonlayer}{background}
					\filldraw[thin,fill=black!50!white,draw=black,opacity=.3] (0,0) circle (\do);
					\fill[thin,fill=white] (0,0) circle (\di);
					\draw[thin,draw=black,opacity=.3] (0,0) circle (\di);
					\filldraw[edge,fill=white!50!black,draw=black](p1) to[bend right=20] (q3) to[bend right=20] (q4) to[bend right=20] (q5) to[bend right=20] (p1);
					\filldraw[edge,fill=white!50!black,draw=black](p2) to[bend right=10] (q6) to[bend right=10] (q9) to[bend right=20] (q10) to[bend right=20] (p2);
					\draw[edge](p3) -- (q11);
					\draw[edge](q1) to[bend right=20] (q2);
					\draw[edge](q7) to[bend right=20] (q8);
					\fill[blue!50!gray,opacity=.5](ql1) circle(6pt);
					\fill[blue!50!gray,opacity=.5](ql2) circle(6pt);
					\fill[green!50!gray,opacity=.5](ql3) circle(6pt);
					\fill[green!50!gray,opacity=.5](ql4) circle(6pt);
					\fill[green!50!gray,opacity=.5](ql5) circle(6pt);
					\fill[red!50!gray,opacity=.5](ql6) circle(6pt);
					\fill[red!50!gray,opacity=.5](ql7) circle(6pt);
					\fill[red!50!gray,opacity=.5](ql8) circle(6pt);
					\fill[red!50!gray,opacity=.5](ql9) circle(6pt);
					\fill[red!50!gray,opacity=.5](ql10) circle(6pt);
					\draw[blue!50!gray](3.2,-.8) node{$\leq 1$};
					\draw[green!50!gray](-.5,-3.5) node{$\leq 2$};
					\draw[red!50!gray](-1.3,3) node{$\leq 3$};
				\end{pgfonlayer}
			\end{tikzpicture}
			\caption{The annular diagram of the non-crossing $3$-partition from Figure~\ref{fig:arc_diagram_noncrossing}.}
			\label{fig:annular_diagram_noncrossing}
		\end{subfigure}
		\hspace*{.5cm}
		\begin{subfigure}[t]{.45\textwidth}
			\centering
			\begin{tikzpicture}[scale=.8]\small
				\def\di{.75};
				\def\do{2.5};
				\foreach \i in {1,...,3} {
					\coordinate (p\i) at ($(-\i*120+20:\di)$);
					\coordinate (pl\i) at ($(p\i) - (-\i*120+20:\di/2.5)$);
					\filldraw[draw=black,fill=white] (p\i) circle(2pt);
					\draw (pl\i) node{\tiny\i};
				}
				\foreach \i in {1,...,11} {
					\coordinate (q\i) at ($(-\i*360/11+360/11:\do)$);
					\coordinate (ql\i) at ($(q\i) + (-\i*360/11+360/11:\do/6)$);
					\filldraw[black] (q\i) circle(2pt);
					\setcounter{lab}{3}
					\addtocounter{lab}{\i}
					\draw (ql\i) node{\tiny\thelab};
				}
				\begin{pgfonlayer}{background}
					\filldraw[thin,fill=black!50!white,draw=black,opacity=.3] (0,0) circle (\do);
					\fill[thin,fill=white] (0,0) circle (\di);
					\draw[thin,draw=black,opacity=.3] (0,0) circle (\di);
					\filldraw[edge,fill=white!50!black,draw=black](p1) to[bend right=20] (q3) to[bend right=20] (q4) to[bend right=20] (q5) to[bend right=20] (p1);
					\filldraw[edge,fill=white!50!black,draw=black](p2) to[bend right=10] (q6) to[bend right=10] (q9) to[bend right=20] (q10) to[bend right=20] (p2);
					\filldraw[edge,fill=white!50!black,draw=black](p3) to[bend right=20] (q11) to[bend right=20] (q1) to[bend right=20] (q2) to[bend right=20] (p3);
					\draw[edge](q7) to[bend right=20] (q8);
					\fill[blue!50!gray,opacity=.5](ql1) circle(6pt);
					\fill[blue!50!gray,opacity=.5](ql2) circle(6pt);
					\fill[green!50!gray,opacity=.5](ql3) circle(6pt);
					\fill[green!50!gray,opacity=.5](ql4) circle(6pt);
					\fill[green!50!gray,opacity=.5](ql5) circle(6pt);
					\fill[red!50!gray,opacity=.5](ql6) circle(6pt);
					\fill[red!50!gray,opacity=.5](ql7) circle(6pt);
					\fill[red!50!gray,opacity=.5](ql8) circle(6pt);
					\fill[red!50!gray,opacity=.5](ql9) circle(6pt);
					\fill[red!50!gray,opacity=.5](ql10) circle(6pt);
					\draw[blue!50!gray](3.2,-.8) node{$\leq 1$};
					\draw[green!50!gray](-.5,-3.5) node{$\leq 2$};
					\draw[red!50!gray](-1.3,3) node{$\leq 3$};
					\draw[blue!50!gray](4,-1.5) node[scale=.5,text width=2.8cm]{not satisfied, because the block contains the white node $3$};
				\end{pgfonlayer}
			\end{tikzpicture}
			\caption{The annular diagram of the crossing $3$-partition from Figure~\ref{fig:arc_diagram_crossing}.}
			\label{fig:annular_diagram_crossing}
		\end{subfigure}
		\caption{Some annular diagrams of $3$-partitions.}
		\label{fig:annular_diagrams}
	\end{figure}

	A $t$-partition is non-crossing if and only if it can be inscribed on an annulus in the previously described manner.  See Figure~\ref{fig:annular_diagram_noncrossing} for an illustration of this construction for the non-crossing $3$-partition 
	\begin{displaymath}
		\bigl\{\{1,6,7,8\},\{2,9,12,13\},\{3,14\},\{4,5\},\{10,11\}\bigr\}.
	\end{displaymath}
	We may ``unfold'' this annular diagram into the arc diagram from Figure~\ref{fig:arc_diagram_noncrossing} by ordering the vertices from $1,2,\ldots,n$ and cutting the arc connecting the biggest with the smallest element in each block.
	
	Note that the crossing $3$-partition 
	\begin{displaymath}
		\bigl\{\{1,6,7,8\},\{2,9,12,13\},\{3,4,5,14\},\{10,11\}\bigr\}
	\end{displaymath}
	can be inscribed on this annulus in a non-crossing fashion, too, see Figure~\ref{fig:annular_diagram_crossing}.  However, it violates the shadow-condition: the block containing $1$ and $6$ casts a shadow on the nodes $4$ and $5$ which makes them invisible to node $3$, while $3$, $4$ and $5$ are in the same block.  Indeed, if we unfold the annular diagram to an arc diagram, then the arc sequences representing the block $\{1,6,7,8\}$ and $\{3,4,5,14\}$ would cross.  See Figure~\ref{fig:arc_diagram_crossing}.

	For work on non-crossing partitions on an annulus without the shadow-condition see \cite{kim13cyclic} and the references contained therein.  
\end{remark}

\subsection{Weighted generating functions and Lagrange Inversion}

Let $x_{1},x_{2},\ldots$ be formal variables, and set $x_{0}\defs 1$.  We define the \defn{weight} of $\pi\in\NC_{n}^{(m)}$ 
by
\begin{displaymath}
	w^{(m)}(\pi) \defs \prod_{i=1}^{\infty}{x_{i}^{\#(\text{blocks of $\pi$ of size $mi$})}}.
\end{displaymath}
For example, the partition displayed in Figure~\ref{fig:arc_diagram_noncrossing} has weight $x_{1}^{3}x_{2}^{2}$, the one in Figure~\ref{fig:arc_diagram_crossing} has weight $x_{1}x_{2}^{3}$.  Moreover, we write $\lvert\pi\rvert$ for the \defn{size} of $\pi$, \ie for the number of elements of the set which is partitioned by $\pi$.  

We first study the case $t=1$, and define 
\begin{displaymath}
	\NC^{(m)}_{\bullet} \defs \bigcup_{n=0}^{\infty}{\NC_{n}^{(m)}},
\end{displaymath}
where, again by definition, $\NC_{0}^{(m)}$ consists of just one element, namely the empty partition.
Moreover, we define the generating function of $m$-divisible non-crossing partitions according to weight by
\begin{displaymath}
	C^{(m)}(z) \defs \sum_{\pi\in\NC^{(m)}_{\bullet}}w^{(m)}(\pi)z^{\lvert \pi\rvert}.
\end{displaymath}

\begin{lemma}\label{lem:nc_gf_by_weight}
	For $m\geq 1$, we have
	\begin{equation}\label{eq:nc_gf_by_weight}
		C^{(m)}(z) = \sum_{i=0}^{\infty}x_{i}z^{mi}C^{(m)}(z)^{mi}.
	\end{equation}
\end{lemma}
\begin{proof}
	Node $1$ must be in one of the blocks, say in a block of size $mi$.  If $v_{1}=1,v_{2},\ldots,v_{mi}$ are all the nodes in this block in linear order, then the nodes between $v_{j}$ and $v_{j+1}$, for $j\in[mi]$, (where we identify $v_{mi+1}$ with $v_{1}$) are involved in a smaller $m$-divisible non-crossing partition.  Equation~\eqref{eq:nc_gf_by_weight} then follows by standard generating function calculus: the block containing $1$ contributes the term $x_{i}z^{mi}$, and the ``small'' $m$-divisible non-crossing partitions between $v_{j}$ and $v_{j+1}$ each contribute a term $C^{(m)}(z)$ for $j\in[mi]$.
\end{proof}

Before we can move on to the case $t>1$ we need to recall the Lagrange Inversion formula.

\begin{lemma}\label{lem:lagrange_inversion}
	Let $f(z)$ be a formal power series with $f(0)=0$ and $f'(0)\neq 0$, and let $F(z)$ be its compositional inverse.  Then, for all integers $a$ and $b$, 
	\begin{equation}\label{eq:lagrange_2}
		\coef{z^{0}}z^{a}F(z)^{b-1} = \coef{z^{0}}z^{b}f(z)^{a-1}f'(z),
	\end{equation}
	and
	\begin{equation}\label{eq:lagrange_1}
		a\coef{z^{0}}z^{a}F(z)^{b} = -b\coef{z^{0}}z^{b}f(z)^{a}.
	\end{equation}
\end{lemma}
\begin{proof}
	Equation~\eqref{eq:lagrange_1} is equivalent to \cite{henrici74applied}*{Theorem~1.9b} or \cite{stanley01enumerative_vol2}*{Theorem~5.4.2}.  The form \eqref{eq:lagrange_2} is not easy to find explicitly in standard books.  However, it can be derived without great effort by ``partial integration'' from \eqref{eq:lagrange_1}:
	\begin{align*}
		b\coef{z^{0}}z^{b}f(z)^{a} & = \coef{z^{-1}}\left(\frac{d}{dz}z^{b}\right)f(z)^{a}\\
		& = \coef{z^{-1}}\left(\frac{d}{dz}\left(\vphantom{\frac{d}{dz}}z^{b}f(z)^{a}\right)-z^{b}\left(\frac{d}{dz}f(z)^{a}\right)\right)\\
		& = -\coef{z^{-1}}z^{b}\left(\frac{d}{dz}f(z)^{a}\right)\\
		& = -a\coef{z^{-1}}z^{b}f(z)^{a-1}f'(z).
	\end{align*}
	The coefficient of $z^{-1}$ in the derivative of a Laurent series must necessarily be zero, which explains the third equality above.
\end{proof}

From now on, let
\begin{equation}\label{eq:F}
	F(z) \defs \frac{z}{\sum_{i=0}^{\infty}x_{i}z^{mi}},
\end{equation}
and let $f(z)$ denote its compositional inverse.  Slightly rewriting \eqref{eq:nc_gf_by_weight}, we see that it is equivalent to
\begin{displaymath}
	z = \frac{zC^{(m)}(z)}{\sum_{i=0}^{\infty}x_{i}\bigl(zC^{(m)}(z)\bigr)^{mi}} = F\bigl(zC^{(m)}(z)\bigr).
\end{displaymath}
Thus we have $zC^{(m)}(z)=f(z)$.

\begin{lemma}\label{lem:tnc_gf_by_weight}
	For $m,n,t\geq 1$ the generating function $\sum_{\pi\in\NC_{n;t}^{(m)}}w^{(m)}(\pi)$ equals
	\begin{equation}\label{eq:tnc_gf_by_weight}
		-\frac{1}{mn-t+1}\coef{z^{t-2}}\left(zF(z)^{-1}-1\right)^{t}\frac{d}{dz}\left(F(z)^{-mn+t-1}\right).
	\end{equation}
\end{lemma}
\begin{proof}
	Let $\pi\in\NC_{n;t}^{(m)}$.  For $j\in[t]$, the element $j$ is contained in a block of size $mi_{j}$ of $\pi$, for some $i_{j}\geq 1$.  These blocks contribute a weight of $x_{i_{j}}z^{mi_{j}}$, for $j\in[t]$, to the generating function of $m$-divisible non-crossing $t$-partitions by weight.  
	
	Similar to the argument in the proof of Lemma~\ref{lem:nc_gf_by_weight}, between the elements of these blocks in the range $\{t{+}1,\ldots,mn\}$ we find $m(i_{1}+\cdots+i_{t})-t+1$ ``small'' $m$-divisible non-crossing $t$-partitions, each contributing a term $C^{(m)}(z)$ to the generating function.  We thus have
	\begin{align*}
		\sum_{\pi\in\NC_{n;t}^{(m)}}w^{(m)}(\pi) & = \coef{z^{mn}}\sum_{i_{1}+\cdots+i_{t}\geq 1}x_{i_{1}}\cdots x_{i_{t}}z^{m(i_{1}+\cdots+i_{t})}C^{(m)}(z)^{m(i_{1}+\cdots+i_{t})-t+1}\\
		& = \coef{z^{0}}\sum_{i_{1}+\cdots+i_{t}\geq 1}x_{i_{1}}\cdots x_{i_{t}}z^{-mn+t-1}f(z)^{m(i_{1}+\cdots+i_{t})-t+1}.
	\end{align*}
	
	If we use \eqref{eq:lagrange_2} with $b=m(i_{1}+\cdots+i_{t})-t+2$, $a=-mn+t-1$ and the roles of $f$ and $F$ interchanged, then this expression becomes
	\begin{multline*}
		\coef{z^{0}}\sum_{i_{1}+\cdots+i_{t}\geq 1}x_{i_{1}}\cdots x_{i_{t}}z^{m(i_{1}+\cdots+i_{t})-t+2}F(z)^{-mn+t-2}F'(z)\\
		= \coef{z^{0}}\left(zF(z)^{-1}-1\right)^{t}z^{-t+2}F(z)^{-mn+t-2}F'(z),
	\end{multline*}
	which is equivalent to the expression in the statement of the lemma.
\end{proof}

\section{Rank enumeration in $\Bigl(\NC_{n;t}^{(m)},\dref\Bigr)$}
	\label{sec:rank_enumeration}
In this section we eventually prove Theorem~\ref{thm:nt_rank_enumeration}, and we prepare this proof with a series of auxiliary results each touching a different aspect of the chain-enumeration in the poset $\Bigl(\NC_{n;t}^{(m)},\dref\Bigr)$.  We start with the enumeration of the elements in $\NC_{n;t}^{(m)}$ with respect to block structure and rank, and conclude the cardinality of $\NC_{n;t}^{(m)}$.

\begin{theorem}\label{thm:tnc_first}
	Let $m,n,t$ be positive integers, and let $s,b_{1},b_{2},\ldots,b_{n}$ be non-negative integers.  The number of $m$-divisible non-crossing partitions in $\NC_{n;t}^{(m)}$ with exactly $b_{i}$ blocks of size $mi$, for $i\in[n]$, equals
	\begin{multline}\label{eq:tncblock}
		\frac{(mn-t+1)(b_{1}+b_{2}+\cdots+b_{n})-mn(b_{1}+b_{2}+\cdots+b_{n}-t)}{(mn-t+1)(b_{1}+b_{2}+\cdots+b_{n})}\\
			\times \binom{mn-t+1}{b_{1}+b_{2}+\cdots+b_{n}-t}\binom{b_{1}+b_{2}+\cdots+b_{n}}{b_{1},b_{2},\ldots,b_{n}}.
	\end{multline}
	Furthermore, the number of $m$-divisible non-crossing partitions in $\NC_{n;t}^{(m)}$ of rank $s$ equals
	\begin{equation}\label{eq:tncrk}
		\frac{mnt-(n-s)(t-1)}{n(mn-t+1)}\binom{mn-t+1}{n-s-t}\binom{n}{s},
	\end{equation}
	and the total number of $m$-divisible non-crossing partitions in $\NC_{n;t}^{(m)}$ is
	\begin{equation}\label{eq:tncsize}
		\frac{mt+1}{mn+1}\binom{(m+1)n-t}{n-t}.
	\end{equation}
\end{theorem}
\begin{proof}
	For proving \eqref{eq:tncblock}, we must extract the coefficient of 
	\begin{equation}\label{eq:coefbi}
		\prod_{i=1}^{n}{x_{i}^{b_{i}}}
	\end{equation}
	in $\sum_{\pi\in\NC_{n;t}^{(m)}}w^{(m)}(\pi)$.  We substitute the definition \eqref{eq:F} of $F(z)$ in \eqref{eq:tnc_gf_by_weight} to obtain
	\begin{align}\label{eq:xi=x}
		-\frac{1}{mn-t+1} & \coef{z^{t-2}}\left(\sum_{i=1}^{\infty}x_{i}z^{mi}\right)^{t}\frac{d}{dz}\left(z^{-mn+t-1}\left(1+\sum_{i=1}^{\infty}x_{i}z^{mi}\right)^{mn-t+1}\right)\\
		& = \coef{z^{t-2}}\left(\sum_{i=1}^{\infty}x_{i}z^{mi}\right)^{t}z^{-mn+t-2}\left(1+\sum_{i=1}^{\infty}x_{i}z^{mi}\right)^{mn-t+1}\notag\\
		& \kern1cm -\coef{z^{t-2}}\left(\sum_{i=1}^{\infty}x_{i}z^{mi}\right)^{t}z^{-mn+t-1}\notag\\
		& \kern2cm \cdot\left(1+\sum_{i=0}^{\infty}x_{i}z^{mi}\right)^{mn-t}\frac{d}{dz}\left(\sum_{i=1}^{\infty}x_{i}z^{mi}\right)\notag\\
		& = \coef{z^{mn}}\sum_{\ell=0}^{mn-t+1}\binom{mn-t+1}{\ell}\left(\sum_{i=1}^{\infty}x_{i}z^{mi}\right)^{t+\ell}\notag\\
		& \kern1cm -\coef{z^{-1}}z^{-mn}\sum_{\ell=0}^{mn-t}\binom{mn-t}{\ell}\left(\sum_{i=1}^{\infty}x_{i}z^{mi}\right)^{t+\ell}\frac{d}{dz}\left(\sum_{i=1}^{\infty}x_{i}z^{mi}\right)\notag\\
		& = \coef{z^{mn}}\sum_{\ell=0}^{mn-t+1}\binom{mn-t+1}{\ell}\left(\sum_{i=1}^{\infty}x_{i}z^{mi}\right)^{t+\ell}\notag\\
		& \kern1cm -\coef{z^{-1}}z^{-mn}\sum_{\ell=0}^{mn-t}\frac{1}{t+\ell+1}\binom{mn-t}{\ell}\frac{d}{dz}\left(\sum_{i=1}^{\infty}x_{i}z^{mi}\right)^{t+\ell+1}.\notag
	\end{align}	
	To the last term, we apply the easily verified fact 
	\begin{equation}\label{eq:dz-f-g}
		\coef{z^{-1}}f(z)\frac{d}{dz}\left(\vphantom{\frac{d}{dz}}g(z)\right) = -\coef{z^{-1}}\frac{d}{dz}\left(\vphantom{\frac{d}{dz}}f(z)\right)g(z).
	\end{equation}
		Thus, the above expression is turned into
	\begin{align*}
		\coef{z^{mn}} & \sum_{\ell=0}^{mn-t+1}\binom{mn-t+1}{\ell}\left(\sum_{i=1}^{\infty}{x_{i}z^{mi}}\right)^{t+\ell}\\
		& \kern1cm -\coef{z^{-1}}z^{-mn-1}\sum_{\ell=0}^{mn-t}\binom{mn-t}{\ell}\frac{mn}{t+\ell+1}\left(\sum_{i=1}^{\infty}{x_{i}z^{mi}}\right)^{t+\ell+1}\\
		& = \coef{z^{mn}}\sum_{\ell=0}^{mn-t+1}\left(\binom{mn-t+1}{\ell}-\binom{mn-t}{\ell-1}\frac{mn}{t+\ell}\right)\left(\sum_{i=1}^{\infty}{x_{i}z^{mi}}\right)^{t+\ell}\\
		& = \coef{z^{mn}}\sum_{\ell=0}^{mn-t+1}\binom{mn-t+1}{\ell}\frac{(mn-t+1)(t+\ell)-mn\ell}{(mn-t+1)(t+\ell)}\left(\sum_{i=1}^{\infty}{x_{i}z^{mi}}\right)^{t+\ell}.
	\end{align*}
	
	As we said in the beginning, we must extract the coefficient of \eqref{eq:coefbi} in this expression.  There is a single summand (in the sum over $\ell$ above) where we find this term: the one with $\ell=b_{1}+b_{2}+\cdots+b_{n}-t$, and its coefficient is \eqref{eq:tncblock}.
	
	In order to establish \eqref{eq:tncrk} we use Corollary~\ref{cor:t_nc_graded} to see that the rank of $\pi\in\NC_{n;t}^{(m)}$ with block structure described in the statement of the theorem is
	\begin{equation}\label{eq:rk}
		\rk(\pi) = n - (b_{1}+b_{2}+\cdots+b_{n}).
	\end{equation}
	Thus, what we have to do is to set all $x_{i}$'s equal to $x$ and extract the coefficient of $x^{n-s}$ in \eqref{eq:tnc_gf_by_weight}.  We start with the equivalent expression \eqref{eq:xi=x}, with the $x_{i}$'s specialised to $x$:
	\begin{multline*}
		-\frac{1}{mn-t+1}  \coef{z^{t-2}}\left(\sum_{i=1}^{\infty}{xz^{mi}}\right)^{t}\frac{d}{dz}\left(z^{-mn+t-1}\left(1+\sum_{i=1}^{\infty}{xz^{mi}}\right)^{mn-t+1}\right)\\
		 = -\frac{1}{mn-t+1}\coef{z^{t-2}}x^{t}\left(\frac{z^{m}}{1-z^{m}}\right)^{t}\frac{d}{dz}\left(z^{-mn+t-1}\left(1+x\frac{z^{m}}{1-z^{m}}\right)^{mn-t+1}\right)
	\end{multline*}
	The coefficient of $x^{n-s}$ in this expression is
	\begin{align*}
		-\frac{1}{mn-t+1} & \coef{z^{t-2}}\left(\frac{z^{m}}{1-z^{m}}\right)^{t}\frac{d}{dz}\left(z^{-mn+t-1}\binom{mn-t+1}{n-s-t}\left(\frac{z^{m}}{1-z^{m}}\right)^{n-s-t}\right)\\
		& \kern-1.2cm = \coef{z^{t-2}}z^{-mn+t-2}\binom{mn-t+1}{n-s-t}\left(\frac{z^{m}}{1-z^{m}}\right)^{n-s}\\
		& \kern-.3cm -\frac{n-s-t}{mn-t+1}\coef{z^{t-2}}z^{-mn+t-1}\binom{mn-t+1}{n-s-t}\left(\frac{z^{m}}{1-z^{m}}\right)^{n-s-1}\frac{mz^{m-1}}{(1-z^{m})^{2}}\\
		& \kern-1.2cm = \coef{z^{ms}}\binom{mn-t+1}{n-s-t}\left(\frac{1}{1-z^{m}}\right)^{n-s}\\
		& \kern-.3cm -m\coef{z^{ms}}\binom{mn-t}{n-s-t-1}\left(\frac{1}{1-z^{m}}\right)^{n-s+1}\\
		& \kern-1.2cm = \binom{mn-t+1}{n-s-t}\binom{n-1}{s} - m\binom{mn-t}{n-s-t-1}\binom{n}{s}.
	\end{align*}
	This is equivalent to the expression in \eqref{eq:tncrk}.  In order to obtain \eqref{eq:tncsize}, one has to sum this expression over all $s$ using the Chu--Vandermonde summation formula, see \eg \cite{graham94concrete}*{Section~1, (5.27)}, which is then followed by some simplification.
\end{proof}

Next we compute the generating function for the completions of elements in $\NC_{n;t}^{(m)}$ by multi-chains of fixed length.

\begin{proposition}\label{prop:TD}
	Let $l\geq 2$ be a positive integer, and let $s'_{2},s_{3},\ldots,s_{l+1}$ be non-negative integers with $s'_{2}+s_{3}+\cdots+s_{l+1}=n-t$.  The generating function
	\begin{equation}\label{eq:CF}
		{\sum}{}^{\displaystyle\prime}w^{(m)}(\pi_{1}),
	\end{equation}
	where the sum is over multi-chains $\pi_{1}\dref\pi_{2}\dref\cdots\dref\pi_{l}$ of $m$-divisible non-crossing partitions in $\NC_{n;t}^{(m)}$, where the rank of $\pi_{i}$ is $s'_{2}+s_{3}+\cdots+s_{i}$, for $i\in\{2,3,\ldots,l\}$, is given by
	\begin{multline}\label{eq:CG}
		\frac{t(mn-t+1)-s_{l+1}(t-1)}{mn-t+1}\binom{mn}{s_{3}}\cdots\binom{mn}{s_{l}}\binom{mn-t+1}{s_{l+1}}\\
		\times\sum_{k=0}^{s_{3}+\cdots+s_{l+1}+t-1}{\frac{1}{mn+k+1}(-1)^{k+s_{3}+\cdots+s_{l+1}+t-1}\binom{s_{3}+\cdots+s_{l+1}+t-1}{k}}\\
		\cdot\coef{z^{0}}z^{k+1}F(z)^{-mn-k-1}.
	\end{multline}
	If $l=2$, the empty product $\binom{mn}{s_{3}}\cdots\binom{mn}{s_{l}}$ of binomial coefficients has to be interpreted as $1$.
\end{proposition}
\begin{proof}
	We prove the assertion by induction on $l$.
	
	For the start of the induction, we let $l=2$ and we return to \eqref{eq:tnc_gf_by_weight}, which provides the generating function for $m$-divisible non-crossing partitions $\pi_{2}$ in $\NC_{n;t}^{(m)}$.
	
	We consider the relation $\pi_{1}\dref\pi_{2}$. How does $\pi_{1}$ arise from $\pi_{2}$?  Clearly, what we have to do is to split some of the blocks of $\pi_{2}$ (possibly all) into smaller blocks to obtain $\pi_{1}$.  More precisely, \emph{each} block of $\pi_{2}$ of size $mb$ is replaced by an $m$-divisible non-crossing partition of $mb$ elements.  Hence, to model this in the generating function, we will replace in \eqref{eq:tnc_gf_by_weight} the variable $x_{b}$ by the coefficient of $u^{mb}$ in $xC^{(m)}(u)$.  (Recall that $C^{(m)}(u)$ is the generating function for $m$-divisible non-crossing partitions.)  Here, the variable $x$ keeps track of the number of blocks of $\pi_{2}$ and, thus, by the equation 
	\begin{displaymath}
		n-\block(\pi_{2}) = \rk(\pi_{2}) = s'_{2},
	\end{displaymath}
	respectively equivalently,
	\begin{displaymath}
		\block(\pi_{2}) = s_{3}+t
	\end{displaymath}
	of the rank of $\pi_{2}$.  From the result we will then extract the coefficient of $x^{s_{3}+t}$ to obtain the generating function \eqref{eq:CF}.  
	
	Now, the substitution of $\coef{u^{mb}}xC^{(m)}(u)$ for $x_{b}$, for $b\geq 1$, in the series $F(z)$ yields the expression
	\begin{align}
		z\left(1+\sum_{b=1}^{\infty}{z^{mb}\coef{u^{mb}}xC^{(m)}(u)}\right)^{-1} & = z\left(1+x(C^{(m)}(z)-1)\right)^{-1}\notag\\
		& = z\left(1+x\left(\frac{f(z)}{z}-1\right)\right)^{-1}.\label{eq:CHa}
	\end{align}

	Therefore, doing this substitution in \eqref{eq:tnc_gf_by_weight}, we arrive at
	\begin{displaymath}
		-\frac{1}{mn-t+1}\coef{z^{t-2}}x^{t}\left(\frac{f(z)}{z}-1\right)^{t}\frac{d}{dz}\left(z^{-mn+t-1}\left(1+x\left(\frac{f(z)}{z}-1\right)\right)^{mn-t+1}\right).
	\end{displaymath}

	Extracting the coefficient of $x^{s_{3}+t}$, we obtain
	\begin{align*}
		- & \frac{1}{mn-t+1} \coef{z^{t-2}}\left(\frac{f(z)}{z}-1\right)^{t}\frac{d}{dz}\left(\binom{mn-t+1}{s_{3}}z^{-mn+t-1}\left(\frac{f(z)}{z}-1\right)^{s_{3}}\right)\\
		& = \coef{z^{t-2}}\binom{mn-t+1}{s_{3}}z^{-mn+t-2}\left(\frac{f(z)}{z}-1\right)^{s_{3}+t}\\
		& \kern.6cm -\frac{s_{3}}{mn-t+1}\coef{z^{t-2}}\binom{mn-t+1}{s_{3}}z^{-mn+t-1}\left(\frac{f(z)}{z}-1\right)^{s_{3}+t-1}\frac{d}{dz}\left(\frac{f(z)}{z}-1\right)\\
		& = \coef{z^{mn}}\binom{mn-t+1}{s_{3}}\left(\frac{f(z)}{z}-1\right)^{s_{3}+t}\\
		& \kern.6cm -\frac{s_{3}}{(mn-t+1)(s_{3}+t)}\coef{z^{-1}}\binom{mn-t+1}{s_{3}}z^{-mn}\frac{d}{dz}\left(\frac{f(z)}{z}-1\right)^{s_{3}+t}.
	\end{align*}

	At this point, we use again \eqref{eq:dz-f-g} to see that the above expression can be rewritten as 
	\begin{align*}
		\coef{z^{mn}} & \binom{mn-t+1}{s_{3}}\left(\frac{f(z)}{z}-1\right)^{s_{3}+t}\\
		& -\frac{mns_{3}}{(mn-t+1)(s_{3}+t)}\coef{z^{-1}}\binom{mn-t+1}{s_{3}}z^{-mn-1}\left(\frac{f(z)}{z}-1\right)^{s_{3}+t}\\
		& = \frac{t(mn-t+1)-s_{3}(t-1)}{(mn-t+1)(s_{3}+t)}\binom{mn-t+1}{s_{3}}\coef{z^{mn}}\left(\frac{f(z)}{z}-1\right)^{s_{3}+t}\\
		& = \frac{t(mn-t+1)-s_{3}(t-1)}{(mn-t+1)(s_{3}+t)}\binom{mn-t+1}{s_{3}}\\
		& \kern4cm \times\sum_{k=0}^{s_{3}+t}(-1)^{s_{3}+t+k}{\binom{s_{3}+t}{k}\coef{z^{mn}}\left(\frac{f(z)}{z}\right)^{k}},
	\end{align*}
	which is seen to be equivalent to \eqref{eq:CG} with $l=2$, once we use \eqref{eq:lagrange_1} with $a=k$ and $b=-mn-k$ and subsequently replace the summation index $k$ by $k+1$.

	\medskip
	
	Next we perform the induction step.  We assume that the assertion is true for multi-chains consisting of $l{-}1$ elements and consider the multi-chain $\pi_{1}\dref\pi_{2}\dref\cdots\dref\pi_{l}$.  The induction hypothesis implies that the generating function
	\begin{displaymath}
		{\sum}{}^{\displaystyle\prime}w^{(m)}(\pi_{2}),
	\end{displaymath}
	where the sum is over all multi-chains $\pi_{2}\dref\cdots\dref\pi_{l}$ of $m$-divisible non-crossing $t$-partitions, where the rank of $\pi_{i}$ is $(s'_{2}+s_{3})+s_{4}+\cdots+s_{i}$, for $i\in\{3,\ldots,l\}$, is given by 
	\begin{multline}\label{eq:CH}
		\frac{t(mn-t+1)-s_{l+1}(t-1)}{mn-t+1}\binom{mn}{s_{4}}\cdots\binom{mn}{s_{l}}\binom{mn-t+1}{s_{l+1}}\\
		\times\sum_{k=0}^{s_{4}+\cdots+s_{l+1}+t-1}{\frac{1}{mn+k+1}(-1)^{k+s_{4}+\cdots+s_{l+1}+t-1}\binom{s_{4}+\cdots+s_{l+1}+t-1}{k}}\\
		\cdot\coef{z^{0}}z^{k+1}F(z)^{-mn-k-1}.
	\end{multline}

	We now consider the relation $\pi_{1}\dref\pi_{2}$.  We already saw that, in order to model this relation in the generating function, we have to replace the variable $x_{b}$ by the coefficient of $u^{mb}$ in $xC^{(m)}(u)$ everywhere.  Again, the variable $x$ keeps track of the number of blocks of $\pi_{2}$.  By Corollary~\ref{cor:t_nc_graded} we have
	\begin{displaymath}
		n-\block(\pi_{2}) = \rk(\pi_{2}) = s'_{2},
	\end{displaymath}
	respectively equivalently, 
	\begin{displaymath}
		\block(\pi_{2}) = s_{3}+s_{4}+\cdots+s_{l+1}+t.
	\end{displaymath}
	Hence, from the result of the substitution we will then extract the coefficient of\break 
$x^{s_{3}+s_{4}+\cdots+s_{l+1}+t}$ to obtain the generating function \eqref{eq:CF}.  

	Using \eqref{eq:CHa}, we see that the substitution of $\coef{u^{b}}xC^{(m)}(u)$ for $x_{b}$ in \eqref{eq:CH} leads to
	\begin{multline*}
		\frac{t(mn-t+1)-s_{l+1}(t-1)}{mn-t+1}\binom{mn}{s_{4}}\cdots\binom{mn}{s_{l}}\binom{mn-t+1}{s_{l+1}}\\
		\times\sum_{k=0}^{s_{4}+\cdots+s_{l+1}+t-1}{\frac{1}{mn+k+1}(-1)^{k+s_{4}+\cdots+s_{l+1}+t-1}\binom{s_{4}+\cdots+s_{l+1}+t-1}{k}}\\
		\cdot\coef{z^{0}}z^{-mn}\left(1+x\left(\frac{f(z)}{z}-1\right)\right)^{mn+k+1}.
	\end{multline*}
	The coefficient of $x^{s_{3}+s_{4}+\cdots+s_{l+1}+t}$ in this expression equals 
	\begin{align}
		& \frac{t(mn-t+1)-s_{l+1}(t-1)}{mn-t+1}\binom{mn}{s_{4}}\cdots\binom{mn}{s_{l}}\binom{mn-t+1}{s_{l+1}}\notag\\
		& \kern1cm \times\sum_{k=0}^{s_{4}+\cdots+s_{l+1}+t-1}{\frac{1}{mn+k+1}(-1)^{k+s_{4}+\cdots+s_{l+1}+t-1}\binom{s_{4}+\cdots+s_{l+1}+t-1}{k}}\notag\\
		& \kern3cm \cdot\coef{z^{0}}z^{-mn}\binom{mn+k+1}{s_{3}+s_{4}+\cdots+s_{l+1}+t}\left(\frac{f(z)}{z}-1\right)^{s_{3}+s_{4}+\cdots+s_{l+1}+t}\notag\\
		& = \frac{t(mn-t+1)-s_{l+1}(t-1)}{(mn-t+1)(s_{3}+s_{4}+\cdots+s_{l+1}+t)}\binom{mn}{s_{4}}\cdots\binom{mn}{s_{l}}\binom{mn-t+1}{s_{l+1}}\notag\\
		& \kern.6cm \times\sum_{k=0}^{s_{4}+\cdots+s_{l+1}+t-1}{(-1)^{k+1}\binom{mn+k}{s_{3}+s_{4}+\cdots+s_{l+1}+t-1}\binom{s_{4}+\cdots+s_{l+1}+t-1}{k}}\notag\\
		& \kern2cm \cdot\sum_{\ell=0}^{s_{3}+s_{4}+\cdots+s_{l+1}+t}{(-1)^{s_{3}+\ell}\binom{s_{3}+s_{4}+\cdots+s_{l+1}+t}{\ell}\cdot\coef{z^{0}}z^{-mn-\ell}f(z)^{\ell}}.\label{eq:CI}
	\end{align}
	
	At this point we should note that the sums over $k$ and $\ell$ have become completely independent.  Moreover, the sum over $k$ can be evaluated by means of the Chu--Vandermonde summation formula.  Namely, we have
	\begin{align*}
		& \sum_{k=0}^{s_{4}+\cdots+s_{l+1}+t-1}{(-1)^{k}\binom{mn+k}{s_{3}+\cdots+s_{l+1}+t-1}\binom{s_{4}+\cdots+s_{l+1}+t-1}{k}}\\
		& \kern.9cm = (-1)^{mn-s_{3}-\cdots-s_{l+1}-t+1}\sum_{k=0}^{s_{4}+\cdots+s_{l+1}+t-1}{\binom{-s_{3}-\cdots-s_{l+1}-t}{mn+k-s_{3}-\cdots-s_{l+1}-t+1}}\\
		& \kern8cm \cdot \binom{s_{4}+\cdots+s_{l+1}+t-1}{s_{4}+\cdots+s_{l+1}+t-1-k}\\
		& \kern.9cm = (-1)^{mn-s_{3}-\cdots-s_{l+1}-t+1}\binom{-s_{3}-1}{mn-s_{3}}\\
		& \kern.9cm = (-1)^{s_{4}+\cdots+s_{l+1}+t-1}\binom{mn}{mn-s_{3}}.
	\end{align*}
	If we substitute this in \eqref{eq:CI} and use \eqref{eq:lagrange_1} with $a=\ell$ and $b=-mn-\ell$, and subsequently replace the summation index $k$ by $k+1$, then we obtain exactly \eqref{eq:CG}.
\end{proof}

We are now in the position to prove the rank enumeration with prescribed block structure of the first element in the multi-chain.

\begin{theorem}\label{thm:nt_rank_enumeration_by_blocks}
	Let $m,n,t,l$ be positive integers, and let $s_{1},s_{2},\ldots,s_{l+1},b_{1},b_{2},\ldots,b_{n}$ be non-negative integers with $s_{1}+s_{2}+\cdots+s_{l+1}=n-t$.  The number of multi-chains $\pi_{1}\dref\pi_{2}\dref\cdots\dref\pi_{l}$ in the poset of $m$-divisible non-crossing $t$-partitions with the property that $\rk(\pi_{i})=s_{1}+s_{2}+\cdots+s_{i}$, for $i\in[l]$, and that the number of blocks of size $mi$ of $\pi_{1}$ is $b_{i}$, for $i\in[n]$, is given by
	\begin{multline}\label{eq:CJ}
		\frac{t(mn-t+1)-s_{l+1}(t-1)}{(mn-t+1)(b_{1}+b_{2}+\cdots+b_{n})}\binom{b_{1}+b_{2}+\cdots+b_{n}}{b_{1},b_{2},\ldots,b_{n}}\\
			\times \binom{mn}{s_{2}}\binom{mn}{s_{3}}\cdots\binom{mn}{s_{l}}\binom{mn-t+1}{s_{l+1}},
	\end{multline}
	if $b_{1}+2b_{2}+\cdots+nb_{n}=n$ and $s_{1}+b_{1}+b_{2}+\cdots+b_{n}=n$, and $0$ otherwise.
\end{theorem}
\begin{proof}
	We use Proposition~\ref{prop:TD} with $s'_{2}=s_{1}+s_{2}$.  We have
	\begin{displaymath}
		n-\block(\pi_{1}) = \rk(\pi_{1}) = s_{1},
	\end{displaymath}
	or equivalently,
	\begin{displaymath}
		\block(\pi_1) = s_{2}+s_{3}+\cdots+s_{l+1}+t.
	\end{displaymath}
	Hence, we must replace $x_{b}$ by $xx_{b}$ in \eqref{eq:CG} and extract the coefficient of
	\begin{displaymath}
		x^{s_{2}+s_{3}+\cdots+s_{l+1}+t}x_{1}^{b_{1}}x_{2}^{b_{2}}\cdots x_{n}^{b_{n}}.
	\end{displaymath}
	
	Doing the substitution $x_{b}\to xx_{b}$ in \eqref{eq:CG}, we obtain
	\begin{multline*}
		\frac{t(mn-t+1)-s_{l+1}(t-1)}{mn-t+1}\binom{mn}{s_{3}}\cdots\binom{mn}{s_{l}}\binom{mn-t+1}{s_{l+1}}\\
		\times\sum_{k=0}^{s_{3}+\cdots+s_{l+1}+t-1}{\frac{1}{mn+k+1}(-1)^{k+s_{3}+\cdots+s_{l+1}+t-1}\binom{s_{3}+\cdots+s_{l+1}+t-1}{k}}\\
		\cdot\coef{z^{mn}}\left(1+x\sum_{i=1}^{\infty}{x_{i}z^{mi}}\right)^{mn+k+1}.
	\end{multline*}
	Here we have to extract the coefficient of $x^{s_{2}+s_{3}+\cdots+s_{l+1}+t}$.  This is
	\begin{multline*}
		\frac{t(mn-t+1)-s_{l+1}(t-1)}{mn-t+1}\binom{mn}{s_{3}}\cdots\binom{mn}{s_{l}}\binom{mn-t+1}{s_{l+1}}\\
		\times\sum_{k=0}^{s_{3}+\cdots+s_{l+1}+t-1}{\frac{1}{mn+k+1}(-1)^{k+s_{3}+\cdots+s_{l+1}+t-1}\binom{s_{3}+\cdots+s_{l+1}+t-1}{k}}\\
		\cdot\coef{z^{mn}}\binom{mn+k+1}{s_{2}+s_{3}+\cdots+s_{l+1}+t}\left(\sum_{i=1}^{\infty}{x_{i}z^{mi}}\right)^{s_{2}+s_{3}+\cdots+s_{l+1}+t}.
	\end{multline*}
	The sum over $k$ is completely analogous to the sum over $k$ in the proof of Proposition~\ref{prop:TD} and, hence, can as well be evaluated by means of the Chu--Vandermonde summation formula.  In this way, the above expression turns into
	\begin{multline}\label{eq:CL}
		\frac{t(mn-t+1)-s_{l+1}(t-1)}{(mn-t+1)(s_{2}+s_{3}+\cdots+s_{l+1}+t)}\binom{mn}{s_{3}}\cdots\binom{mn}{s_{l}}\binom{mn-t+1}{s_{l+1}}\\
			\times \binom{mn}{s_{2}}\coef{z^{mn}}\left(\sum_{i=1}^{\infty}{x_{i}z^{mi}}\right)^{s_{2}+s_{3}+\cdots+s_{l+1}+t}.
	\end{multline}
	
	Now we must extract the coefficient of $x_{1}^{b_{1}}x_{2}^{b_{2}}\cdots x_{n}^{b_{n}}$ in this expression.  The result is
	\begin{multline*}
		\frac{t(mn-t+1)-s_{l+1}(t-1)}{(mn-t+1)(s_{2}+s_{3}+\cdots+s_{l+1}+t)}\binom{mn}{s_{3}}\cdots\binom{mn}{s_{l}}\binom{mn-t+1}{s_{l+1}}\\
			\times \binom{mn}{s_{2}}\binom{s_{2}+s_{3}+\cdots+s_{l+1}+t}{b_{1},b_{2},\ldots,b_{n}},
	\end{multline*}
	or zero if $s_{2}+s_{3}+\cdots+s_{l+1}+t\neq b_{1}+b_{2}+\cdots+b_{n}$.  However, our assumptions do indeed imply this relation.  Little manipulation then leads to the expression \eqref{eq:CJ}.
\end{proof}

We may now embark on the proof of Theorem~\ref{thm:nt_rank_enumeration}.  
	
\begin{proof}[Proof of Theorem~\ref{thm:nt_rank_enumeration}]
	We reuse \eqref{eq:CL}.  Since, at this point, all rank conditions of the multi-chain $\pi_{1}\dref\pi_{2}\dref\cdots\dref\pi_{l}$ are already built in the calculation, it suffices to put all the $x_{i}$'s equal to $1$ and then finish the calculation.  Consequently, we get
	\begin{align*}
		& \frac{t(mn-t+1)-s_{l+1}(t-1)}{(mn-t+1)(s_{2}+s_{3}+\cdots+s_{l+1}+t)}\binom{mn}{s_{3}}\cdots\binom{mn}{s_{l}}\binom{mn-t+1}{s_{l+1}}\\
		& \kern5cm \times \binom{mn}{s_{2}}\coef{z^{mn}}\left(\sum_{i=1}^{\infty}{z^{mi}}\right)^{s_{2}+s_{3}+\cdots+s_{l+1}+t}\\
		& \kern1cm = \frac{t(mn-t+1)-s_{l+1}(t-1)}{(mn-t+1)(s_{2}+s_{3}+\cdots+s_{l+1}+t)}\binom{mn}{s_{3}}\cdots\binom{mn}{s_{l}}\binom{mn-t+1}{s_{l+1}}\\
		& \kern5cm \times \binom{mn}{s_{2}}\coef{z^{mn}}\left(\frac{z^{m}}{1-z^{m}}\right)^{n-s_{1}}\\
		& \kern1cm = \frac{t(mn-t+1)-s_{l+1}(t-1)}{(mn-t+1)(n-s_{1})}\binom{mn}{s_{3}}\cdots\binom{mn}{s_{l}}\binom{mn-t+1}{s_{l+1}}\\
		& \kern5cm \times \binom{mn}{s_{2}}\coef{z^{ms_{1}}}\left(\frac{1}{1-z^{m}}\right)^{n-s_{1}}.
	\end{align*}
	Now the extraction of the coefficient of $z^{ms_{1}}$ in the binomial series and little simplification finishes the proof.
\end{proof}

We record two corollaries of Theorem~\ref{thm:nt_rank_enumeration}.  The first corollary provides a simple formula for the number of maximal chains in $\Bigl(\NC_{n;t}^{(m)},\dref\Bigr)$.

\begin{corollary}\label{cor:nt_max_chains}
	The number of maximal chains in the poset of $m$-divisible non-crossing $t$-partitions equals $t(mn)^{n-t-1}$.
\end{corollary}
\begin{proof}
	Choose $l=n-t$, $s_{1}=s_{2}=\cdots=s_{n-t}=1$, and $s_{n-t+1}=0$ in Theorem~\ref{thm:nt_rank_enumeration}.
\end{proof}

The second corollary allows us to compute the \defn{zeta polynomial} of $\Bigl(\NC_{n;t}^{(m)},\dref\Bigr)$, \ie the polynomial whose evaluation at $l$ yields the total number of multi-chains of length $l-1$.

\begin{corollary}\label{cor:nt_zeta}
	For $l\geq 1$, the number of multi-chains $\pi_{1}\dref\pi_{2}\dref\cdots\dref\pi_{l-1}$ in the poset $\Bigl(\NC_{n;t}^{(m)},\dref\Bigr)$ equals
	\begin{equation}\label{eq:nt_zeta}
		\ZetaPol_{n;t}^{(m)}(l) \defs \frac{(l-1)mt+1}{(l-1)mn+1}\binom{n+(l-1)mn-t}{n-t}.
	\end{equation}
\end{corollary}
\begin{proof}
	In Theorem~\ref{thm:nt_rank_enumeration}, we replace $l$ by $l-1$.  What we have to do is to sum the expression \eqref{eq:nt_rank_enumeration} over all possible $s_{1},s_{2},\ldots,s_{l}$.  We have
	\begin{align*}
		\sum_{s_{1}+s_{2}+\cdots+s_{l}=n-t} & {\frac{t(mn-t+1)-s_{l}(t-1)}{n(mn-t+1)}\binom{n}{s_{1}}\binom{mn}{s_{2}}\cdots\binom{mn}{s_{l-1}}\binom{mn-t+1}{s_{l}}}\\
		& = \sum_{s_{1}+s_{2}+\cdots+s_{l}=n-t}{\frac{t}{n}\binom{n}{s_{1}}\binom{mn}{s_{2}}\cdots\binom{mn}{s_{l-1}}\binom{mn-t+1}{s_{l}}}\\
		& \kern1cm - \sum_{s_{1}+s_{2}+\cdots+s_{l}=n-t}{\frac{t-1}{n}\binom{n}{s_{1}}\binom{mn}{s_{2}}\cdots\binom{mn}{s_{l-1}}\binom{mn-t}{s_{l}-1}}.
	\end{align*}
	Both multiple sums are iterated Chu--Vandermonde convolutions.  Thus, we obtain
	\begin{multline*}
		\frac{t}{n}\binom{n+(l-1)mn-t+1}{n-t} - \frac{t-1}{n}\binom{n+(l-1)mn-t}{n-t-1}\\
			= \frac{t}{n}\binom{n+(l-1)mn-t}{n-t} + \frac{1}{n}\binom{n+(l-1)mn-t}{n-t-1},
	\end{multline*}
	which can be simplified to the right-hand side of \eqref{eq:nt_zeta}.
\end{proof}

\section{The $M$-triangle of $\NC_{n;t}^{(m)}$}
	\label{sec:m_triangle}
We now consider the (bivariate) generating function of the M{\"o}bius function in\break ${\Bigl(\NC_{n;t}^{(m)},\dref\Bigr)}$; see \cite{stanley11enumerative_vol1}*{Section~3.7} for a definition of and some background on the M{\"o}bius function of posets in general.  

\begin{definition}\label{def:m_triangle}
	For positive integers $m,n,t$, we define the \defn{$M$-triangle} of $\NC_{n;t}^{(m)}$ by
	\begin{equation}\label{eq:m_triangle}
		M_{n;t}^{(m)}(x,y) \defs \sum_{\pi_{1},\pi_{2}\in\NC_{n;t}^{(m)}}{\mu(\pi_{1},\pi_{2})x^{\rk(\pi_{1})}y^{\rk(\pi_{2})}},
	\end{equation}
	where $\mu$ denotes the M{\"o}bius function of $\bigl(\NC_{n;t}^{(m)},\dref\bigr)$.
\end{definition}

If $t=1$, then we recover the type-$A$ case of \cite{armstrong09generalized}*{Definition~5.3.1}.  We may use Theorem~\ref{thm:nt_rank_enumeration} to explicitly compute $M_{n;t}^{(m)}$.

\begin{theorem}\label{thm:nt_mtriangle}
	For positive integers $m,n,t$, the $M$-triangle of $\NC_{n;t}^{(m)}$ equals
	\begin{align}
		M_{n;t}^{(m)}(x,y) & = \sum_{r=0}^{n-t}\sum_{s=r}^{n-t}{(-1)^{s-r}\frac{t(mn-t+1)-(n-t-s)(t-1)}{n(mn-t+1)}}\notag\\
		& \kern1cm \cdot\binom{n}{r}\binom{mn-t+1}{n-t-s}\binom{mn+s-r-1}{s-r}x^{r}y^{s}\label{eq:nt_mtriangle}.
	\end{align}
\end{theorem}
\begin{proof}
	We follow a strategy that has been applied earlier in \cite{krattenthaler06ftriangle}*{Section~8} and \cite{krattenthaler10decomposition}*{Section~9}, and which exploits the equality
	\begin{displaymath}
		\mu(\pi_{1},\pi_{2}) = \ZetaPol_{n;t}^{(m)}(\pi_{1},\pi_{2};-1),
	\end{displaymath}
	where $\ZetaPol_{n;t}^{(m)}(\pi_{1},\pi_{2};z)$ is the polynomial whose evaluation at positive $z$ yields the number of multi-chains of length $z-1$ which lie weakly between $\pi_{1}$ and $\pi_{2}$.  This equality is \cite{stanley11enumerative_vol1}*{Proposition~3.12.1(c)} tailored to our current situation.  
	
	In order to compute the coefficient of $x^{r}y^{s}$ of $M_{n;t}^{(m)}$, we sum the zeta polynomials $\ZetaPol_{n;t}^{(m)}(\pi_{1},\pi_{2};z)$ over all $\pi_{1},\pi_{2}\in\NC_{n;t}^{(m)}$ with $\rk(\pi_{1})=r$ and $\rk(\pi_{2})=s$, and evaluate the resulting expression at $z=-1$.
	
	The zeta polynomials we require can be obtained from Theorem~\ref{thm:nt_rank_enumeration} by setting $l=z+1$, $s_{1}=r$, $n-t-s_{l+1}=s$, $s_{2}+s_{3}+\cdots+s_{l}=s-r$, and then summing over all possible $s_{2},s_{3},\ldots,s_{l}$.  By using the Chu--Vandermonde summation formula, one obtains
	\begin{displaymath}
		\frac{t(mn-t+1)-(n-t-s)(t-1)}{n(mn-t+1)}\binom{n}{r}\binom{mn-t+1}{n-t-s}\binom{zmn}{s-r}.
	\end{displaymath}
	If we evaluate this expression at $z=-1$, then we obtain
	\begin{displaymath}
		(-1)^{s-r}\frac{t(mn-t+1)-(n-t-s)(t-1)}{n(mn-t+1)}\binom{n}{r}\binom{mn-t+1}{n-t-s}\binom{mn+s-r-1}{s-r}
	\end{displaymath}
	as desired.
\end{proof}

The main purpose of our consideration of the $M$-triangle of $\NC_{n;t}^{(m)}$ is a surprising connection conjectured in the case $t=1$ by Chapoton in \cites{chapoton04enumerative,chapoton06sur} for $m=1$, and generalised by Armstrong in \cite{armstrong09generalized}*{Section~5.3} to $m>1$.  

This connection predicts the existence of two polynomials, denoted by $F_{n;t}^{(m)}$ and $H_{n;t}^{(m)}$, with non-negative integer coefficients, that can be obtained from $M_{n;t}^{(m)}$ by certain substitutions of the variables. In the case $t=1$, explicit combinatorial explanations of these polynomials (as generating functions of certain combinatorial objects) are known.  For $t>1$, we conjecture a combinatorial description of one of these polynomials in Section~\ref{sec:h_triangle}.

\begin{theorem}\label{thm:fhm_correspondence}
	For positive integers $m,n,t$, there exist polynomials $F_{n;t}^{(m)},H_{n;t}^{(m)}\in\mathbb{Z}[x,y]$ with non-negative integer coefficients such that the following equalities hold:
	\begin{align*}
		F_{n;t}^{(m)}(x,y) & = y^{n-t}M_{n;t}^{(m)}\left(\frac{y+1}{y-x},\frac{y-x}{y}\right)\\
			& = x^{n-t}H_{n;t}^{(m)}\left(\frac{x+1}{x},\frac{y+1}{x+1}\right),\\
		H_{n;t}^{(m)}(x,y) & = \bigl(x(y-1)+1\bigr)^{n-t}M_{n;t}^{(m)}\left(\frac{y}{y-1},\frac{x(y-1)}{x(y-1)+1}\right)\\
			& = (x-1)^{n-t}F_{n;t}^{(m)}\left(\frac{1}{x-1},\frac{x(y-1)+1}{x-1}\right),\\
		M_{n;t}^{(m)}(x,y) & = (xy-1)^{n-t}F_{n;t}^{(m)}\left(\frac{1-y}{xy-1},\frac{1}{xy-1}\right)\\
			& = (1-y)^{n-t}H_{n;t}^{(m)}\left(\frac{y(x-1)}{1-y},\frac{x}{x-1}\right).
	\end{align*}
\end{theorem}
\begin{proof}
	By Theorem~\ref{thm:nt_mtriangle}, we have
	\begin{align*}
		H_{n;t}^{(m)} & (x,y) \defs \bigl(x(y-1)+1\bigr)^{n-t}M_{n;t}^{(m)}\left(\frac{y}{y-1},\frac{x(y-1)}{1+x(y-1)}\right)\\
		& = \sum_{0\le r\le s\le n-t}{(-1)^{s-r}\frac{t(mn-t+1)-(n-t-s)(t-1)}{n(mn-t+1)}}\\
		& \kern.5cm \cdot\binom{n}{r}\binom{mn-t+1}{n-t-s}\binom{mn+s-r-1}{s-r}x^{s}y^{r}(y-1)^{s-r}\bigl(1+x(y-1)\bigr)^{n-t-s}\\
		& = \sum_{0\le r\le s\le n-t}{(-1)^{s-r}\frac{t(mn-t+1)-(n-t-s)(t-1)}{n(mn-t+1)}}\\
		& \kern.5cm \cdot\binom{n}{r}\binom{mn-t+1}{n-t-s}\binom{mn+s-r-1}{s-r}x^{s}y^{r}(y-1)^{s-r}\\
		& \kern3cm \cdot\sum_{k=0}^{n-t-s}{\binom{n-t-s}{k}x^{n-t-k-s}(y-1)^{n-t-k-s}}\\
		& = \sum_{0\le r,k\le n-t}{\binom{n}{r}\binom{mn-t+1}{k}x^{n-t-k}y^{r}(y-1)^{n-t-k-r}}\\
		& \kern.5cm \times\sum_{s=r}^{n-t-k}{\left(\frac{t(mn-t+1)-k(t-1)}{n(mn-t+1)}-\frac{(n-t-k-s)(t-1)}{n(mn-t+1)}\right)}\\
		& \kern3cm \cdot\binom{-mn}{s-r}\binom{mn-t-k+1}{n-t-k-s}.
\end{align*}
The sum over $s$ can be evaluated by means of the Chu--Vandermonde summation formula. This yields
\begin{align*}
		H_{n;t}^{(m)} & (x,y)= \sum_{0\le r,k\le n-t}{\binom{n}{r}\binom{mn-t+1}{k}x^{n-t-k}y^{r}(y-1)^{n-t-k-r}}\\
		& \kern.5cm \cdot\left(\frac{t(mn-t+1)-k(t-1)}{n(mn-t+1)}\binom{-t-k+1}{n-t-k-r}\right.\\
		& \kern3cm \left.-\frac{(mn-t-k+1)(t-1)}{n(mn-t+1)}\binom{-t-k}{n-t-k-r-1}\right)\\
		& = \sum_{0\le r,k\le n-t}{\binom{n}{r}\binom{mn-t+1}{k}x^{n-t-k}y^{r}\sum_{h=0}^{n-t-k-r}{\binom{n-t-k-r}{h}(-1)^{h}y^{n-t-k-r-h}}}\\
		& \kern.5cm \cdot\left(\frac{t(mn-t+1)-k(t-1)}{n(mn-t+1)}\binom{-t-k+1}{n-t-k-r}\right.\\
		& \kern3cm \left.-\frac{(mn-t-k+1)(t-1)}{n(mn-t+1)}\binom{-t-k}{n-t-k-r-1}\right)\\
		& = \sum_{k=0}^{n-t}\sum_{h=0}^{n-t-k}{(-1)^{h}\binom{mn-t+1}{k}x^{n-t-k}y^{n-t-k-h}\sum_{r=0}^{n-t-k-h}{\binom{n}{r}}}\\
		& \kern.5cm \cdot\left(\frac{t(mn-t+1)-k(t-1)}{n(mn-t+1)}\binom{-t-k+1}{h}\binom{-t-k-h+1}{n-t-k-r-h}\right.\\
		& \kern1cm \left.-\frac{(mn-t-k+1)(t-1)(n-t-k-r)}{n(mn-t+1)h}\binom{-t-k}{h-1}\binom{-t-k-h+1}{n-t-k-r-h}\right)\\
		& = \sum_{k=0}^{n-t}\sum_{h=0}^{n-t-k}{(-1)^{h}\binom{mn-t+1}{k}\binom{-t-k+1}{h}x^{n-t-k}y^{n-t-k-h}}\\
		& \kern.5cm \times\sum_{r=0}^{n-t-k-h}\left(\frac{t(mn-t+1)-k(t-1)}{n(mn-t+1)}\binom{n}{r}\binom{-t-k-h+1}{n-t-k-r-h}\right.\\
		& \kern2cm -\frac{(mn-t-k+1)(t-1)(n-t-k)}{n(mn-t+1)(-t-k+1)}\binom{n}{r}\binom{-t-k-h+1}{n-t-k-r-h}\\
		& \kern2cm \left.-\frac{(mn-t-k+1)(t-1)n}{n(mn-t+1)(t+k-1)}\binom{n-1}{r-1}\binom{-t-k-h+1}{n-t-k-r-h}\right).
\end{align*}
Now the sum over $r$ can be evaluated by means of the Chu--Vandermonde summation formula. Consequently, we obtain
\begin{align}
\notag
		H_{n;t}^{(m)} & (x,y) = \sum_{k=0}^{n-t}\sum_{h=0}^{n-t-k}{(-1)^{h}\binom{mn-t+1}{k}\binom{-t-k+1}{h}x^{n-t-k}y^{n-t-k-h}}\\
\notag
		& \kern.5cm \cdot\left(\frac{t(mn-t+1)-k(t-1)}{n(mn-t+1)}\binom{n-t-k-h+1}{n-t-k-h}\right.\\
\notag
		& \kern2cm +\frac{(mn-t-k+1)(t-1)(n-t-k)}{n(mn-t+1)(t+k-1)}\binom{n-t-k-h+1}{n-t-k-h}\\
\notag
		& \kern2cm \left.-\frac{(mn-t-k+1)(t-1)n}{n(mn-t+1)(t+k-1)}\binom{n-t-k-h}{n-t-k-h-1}\right)\\
\notag
		& = \sum_{k=0}^{n-t}\sum_{h=0}^{n-t-k}{\binom{mn-t+1}{k}\binom{t+k+h-2}{h}x^{n-t-k}y^{n-t-k-h}}\\
\notag
		& \kern4cm \cdot\left(\frac{m(n-k)-t+1}{mn-t+1}-\frac{mhk}{(mn-t+1)(t+k-1)}\right)\\
\notag
		& = \sum_{k=0}^{n-t}\sum_{h=0}^{n-t-k}{\binom{mn-t+1}{k}\binom{t+k+h-2}{h}x^{n-t-k}y^{n-t-k-h}}\\
\notag
		& \kern4cm \cdot\left(1-\frac{mk(t+k+h-1)}{(mn-t+1)(t+k-1)}\right)\\
\notag
		& = \sum_{k=0}^{n-t}\sum_{h=0}^{n-t-k}\left(\binom{mn-t+1}{k}\binom{t+k+h-2}{h}\right.\\
		& \kern2cm \left.-m\binom{mn-t}{k-1}\binom{t+k+h-1}{h}\right)x^{n-t-k}y^{n-t-k-h}.
\label{eq:H}
	\end{align}
	Using the restriction $h\leq n-t-k$, we see that 
	\begin{align*}
		(mn-t+1)(t+k-1) -mk(t+k+h-1) & \ge (mn-t+1)(t+k-1) -mk(n-1)\\
		& \ge mk+(t-1)(mn-k-t+1).
	\end{align*}
Since $t\ge1$ by assumption and $k\le n-t$ due to the restriction on the first sum, the above expression is evidently positive.  This shows that the coefficients of this polynomial are indeed positive integers.

\medskip
	Now, using the expression for $H_{n;t}^{(m)}(x,y)$ found in \eqref{eq:H}, we compute
	\begin{align*}
		F_{n;t}^{(m)} & (x,y) \defs x^{n-t}H_{n;t}^{(m)}\left(\frac{x+1}{x},\frac{y+1}{x+1}\right)\\
		& = \sum_{k=0}^{n-t}\sum_{h=0}^{n-t-k}\left(\binom{mn-t+1}{k}\binom{t+k+h-2}{h}\right.\\
		& \kern2cm \left.-m\binom{mn-t}{k-1}\binom{t+k+h-1}{h}\right)(x+1)^{h}x^{k}(y+1)^{n-t-k-h}\\
		& = \sum_{k=0}^{n-t}\sum_{h=0}^{n-t-k}\left(\binom{mn-t+1}{k}\binom{t+k+h-2}{h}-m\binom{mn-t}{k-1}\binom{t+k+h-1}{h}\right)\\
		& \kern2cm \cdot\sum_{a=0}^h\sum_{b=0}^{n-t-k-h}\binom{h}{a}x^{a+k}\binom{n-t-k-h}{b}y^{b}\\
		& = \sum_{a=0}^{n-t}\sum_{b=0}^{n-t-a}\sum_{k=0}^{n-t-a-b}\sum_{h=a}^{n-t-k-b}(-1)^{n-t-k-b}\binom{mn-t+1}{k}\\
		& \kern3cm \cdot\binom{-t-k+1}{h}\binom{h}{a}\binom{-b-1}{n-t-k-h-b}x^{a+k}y^{b}\\
		& \kern1cm -\sum_{a=0}^{n-t}\sum_{b=0}^{n-t-a}\sum_{k=0}^{n-t-a-b}\sum_{h=a}^{n-t-k-b}(-1)^{n-t-k-b}m\binom{mn-t}{k-1}\\
		& \kern3cm \cdot\binom{-t-k}{h}\binom{h}{a}\binom{-b-1}{n-t-k-h-b}x^{a+k}y^{b}\\
		& = \sum_{a=0}^{n-t}\sum_{b=0}^{n-t-a}\sum_{k=0}^{n-t-a-b}\sum_{h=a}^{n-t-k-b}(-1)^{n-t-k-b}\binom{mn-t+1}{k}\\
		& \kern3cm \cdot\binom{-t-k+1}{a}\binom{-t-k-a+1}{h-a}\binom{-b-1}{n-t-k-h-b}x^{a+k}y^{b}\\
		& \kern1cm -(-1)^{n-t-k-b}m\binom{mn-t}{k-1}\\
		& \kern3cm \cdot\binom{-t-k}{a}\binom{-t-k-a}{h-a}\binom{-b-1}{n-t-k-h-b}x^{a+k}y^{b}.
\end{align*}
The sum over $h$ can be evaluated by means of the Chu--Vandermonde summation formula. If we additionally replace $a$ by $a-k$, then we
obtain
\begin{align*}
	F_{n;t}^{(m)} & (x,y) = \sum_{a=0}^{n-t}\sum_{b=0}^{n-t-a}\sum_{k=0}^{a}\binom{mn-t+1}{k}\binom{t+a-2}{a-k}\binom{n-1}{n-t-a-b}x^{a}y^{b}\\
		& \kern1cm -\sum_{a=0}^{n-t}\sum_{b=0}^{n-t-a}\sum_{k=1}^{a}m\binom{mn-t}{k-1}\binom{t+a-1}{a-k}\binom{n}{n-t-a-b}x^{a}y^{b}.
\end{align*}
Here, the sums over $k$ can be evaluated by means of the Chu--Vandermonde summation formula. Thus, we arrive at
\begin{align}
\notag
		F_{n;t}^{(m)} & (x,y) = \sum_{a=0}^{n-t}\sum_{b=0}^{n-t-a}\left(\binom{mn+a-1}{a}\binom{n-1}{n-t-a-b}\right.\\
\notag
&\kern3cm
\left.
-m\binom{mn+a-1}{a-1}\binom{n}{n-t-a-b}\right)x^{a}y^{b}\\
		& = \sum_{a=0}^{n-t}\sum_{b=0}^{n-t-a}\binom{mn+a-1}{a}\binom{n}{t+a+b}\frac{t+b}{n}x^{a}y^{b}.
\label{eq:Ftr}
\end{align}
	The coefficient of $x^{a}y^{b}$ in $F_{n;t}^{(m)}(x,y)$ is clearly positive.  The remaining equalities follow easily.
\end{proof}

Let us call the polynomials $F_{n;t}^{(m)}$ and $H_{n;t}^{(m)}$ the \defn{$F$-triangle} and the \defn{$H$-triangle}, respectively.  Since Theorem~\ref{thm:fhm_correspondence} states that these polynomials have non-negative integer coefficients, it is an intriguing challenge to explain these polynomials combinatorially.

For $t=1$, such explanations were given in \cite{armstrong09generalized}*{Section~5.3}, generalising the case $m=1$ from \cites{chapoton04enumerative,chapoton06sur}.  The $F$-triangle is the generating function for facets of the $m$-divisible cluster complex, where $x$ marks the coloured positive roots, and $y$ marks the coloured negative simple roots.  The $H$-triangle is the generating function for positive chambers in the extended Shi arrangement, where $x$ marks coloured floors and $y$ marks coloured ceilings.

Starting from these concrete definitions of the $F$-, $H$-, and $M$-triangle, it is far from obvious that the relations from Theorem~\ref{thm:fhm_correspondence} hold; they were shown to be satisfied in several papers~\cites{athanasiadis07on,krattenthaler06ftriangle,krattenthaler06mtriangle,krattenthaler10decomposition,thiel14on,tzanaki08faces}.  (These papers address a more general definition of those triangles for finite Coxeter groups.)

Surprisingly, \cite{garver20chapoton} defines $F$-, $H$-, and $M$-triangles in the context of yet another generalisation of non-crossing partitions.  Conjecture~1.2 and Corollary~5.5 in \cite{garver20chapoton} suggest that the relations from Theorem~\ref{thm:fhm_correspondence} should hold in their setting, too.  

\section{A conjectural combinatorial description of the $H$-triangle}
	\label{sec:h_triangle}
\subsection{Geometric multi-chains of filters of the triangular poset}
For $n\geq 1$, let 
\begin{displaymath}
	T_{n} \defs \bigl\{(i,j)\mid 1\leq i<j\leq n\}
\end{displaymath}
be the set of ordered pairs of integers between $1$ and~$n$.  We define a partial order on $T_{n}$ by setting
\begin{equation}
	(i,j) \preceq (k,l) \quad\text{if and only if}\quad i\geq k\;\text{and}\;j\leq l
\end{equation}
for all $(i,j),(k,l)\in T_{n}$.  The poset $(T_{n},\preceq)$ is the \defn{triangular poset} of degree $n$.  For $(i,j),(k,l)\in T_{n}$ we define their formal sum by
\begin{displaymath}
	(i,j)+(k,l) \defs \begin{cases}(i,l), & \text{if}\;j=k,\\ \bot, & \text{otherwise},\end{cases}
\end{displaymath}
which extends to subsets of $T_{n}$ as follows:
\begin{displaymath}
	A+B \defs \{a+b\mid a\in A,b\in B,a+b\neq\bot\}.
\end{displaymath}

A \defn{filter} (of $T_{n}$) is a set $X\subseteq T_{n}$ such that $(i,j)\in X$ and $(i,j)\preceq (k,l)$ together imply $(k,l)\in X$.  A \defn{multi-chain of filters} (of $T_{n}$) is a tuple $(V_{m},V_{m-1},\ldots,V_{1})$ of filters with $V_{m}\subseteq V_{m-1}\subseteq\cdots\subseteq V_{1}\subseteq T_{n}$.  The next definition is adapted from \cite{athanasiadis05refinement}*{page~180}.  

\begin{definition}\label{def:geometric_multichain}
	Let $V_{0}=T_{n}$.  A multi-chain of filters $(V_{m},V_{m-1},\ldots,V_{1})$ is \defn{geometric} if
	\begin{equation}\label{eq:geometric_1}
		V_{i} + V_{j}\subseteq V_{i+j}
	\end{equation}
	holds for all indices $i,j\ge1$ (where we set $V_{k}=V_{m}$ for $k>m$), and 
	\begin{equation}\label{eq:geometric_2}
		\bigl(T_{n}\setminus V_{i}\bigr) + \bigl(T_{n}\setminus V_{j}\bigr) \subseteq (T_{n}\setminus V_{i+j})
	\end{equation}
	holds for all $i,j\geq 1$ with $i+j\leq m$.
\end{definition}

Now, for $t\in[n]$, consider the filter
\begin{displaymath}
	T_{n;t} \defs \bigl\{(i,j)\in T_{n}\mid j>t\bigr\}.
\end{displaymath}
A \defn{$t$-filter} is a filter of the poset $(T_{n;t},\preceq)$, and a \defn{multi-chain of $t$-filters} is a tuple $(V_{m},V_{m-1},\ldots,V_{1})$ of $t$-filters with $V_{m}\subseteq V_{m-1}\subseteq\cdots\subseteq V_{1}$.

By definition, every multi-chain of $t$-filters is a multi-chain of filters.  We call a multi-chain of $t$-filters \defn{geometric} if it satisfies both \eqref{eq:geometric_1} and \eqref{eq:geometric_2}.  

\begin{conjecture}\label{conj:geometric_chains_cardinality}
	The number of geometric multi-chains of $t$-filters $(V_{m},V_{m-1},\ldots,V_{1})$ is
	\begin{displaymath}
		\frac{mt+1}{mn+1}\binom{(m+1)n-t}{n-t}.
	\end{displaymath}
\end{conjecture}

In fact, it may appear more natural to generalise the notion of geometric multi-chains to $T_{n;t}$ by adapting \eqref{eq:geometric_2} so that we take complements with respect to $T_{n;t}$.  This definition, however, does not produce the desired number of geometric multi-chains as the next example shows.  The number that appears in Conjecture~\ref{conj:geometric_chains_cardinality} is desirable, because it is precisely the cardinality of $\NC_{n;t}^{(m)}$, see Theorem~\ref{thm:tnc_first}.

\begin{example}\label{ex:geometric_1}
	Let $m=2$, $n=3$, and $t=2$.  Then, we have $T_{3;2}=\bigl\{(1,3),(2,3)\bigr\}$, and the three $2$-filters are $\emptyset$, $\bigl\{(1,3)\bigr\}$, and $\bigl\{(1,3),(2,3)\bigr\}$.  There exist the following six multi-chains of $2$-filters of length $2$:
	\begin{displaymath}\begin{aligned}
		& \Bigl(\emptyset,\emptyset\Bigr), && \Bigl(\emptyset,\bigl\{(1,3)\bigr\}\Bigr), && \Bigl(\emptyset,\bigl\{(1,3),(2,3)\bigr\}\Bigr),\\
		& \Bigl(\bigl\{(1,3)\bigr\},\bigl\{(1,3)\bigr\}\Bigr), && \Bigl(\bigl\{(1,3)\bigr\},\bigl\{(1,3),(2,3)\bigr\}\Bigr), && \Bigl(\bigl\{(1,3),(2,3)\bigr\},\bigl\{(1,3),(2,3)\bigr\}\Bigr).
	\end{aligned}\end{displaymath}
	Since we have
	\begin{displaymath}
		(1,3)+(1,3) = (1,3)+(2,3) = (2,3)+(2,3) = \bot,
	\end{displaymath}
	we conclude that each of these multi-chains satisfies \eqref{eq:geometric_1}, and the adapted variant of \eqref{eq:geometric_2}, where complements are taken in $T_{3;2}$.  If we take \eqref{eq:geometric_2} as stated, however, then we observe that $\Bigl(\bigl\{(1,3)\bigr\},\bigl\{(1,3)\bigr\}\Bigr)$ does not satisfy this condition, because 
	\begin{displaymath}
		A = T_{3}\setminus\bigl\{(1,3)\bigr\}=\bigl\{(1,2),(2,3)\bigr\},
	\end{displaymath}
	and $A+A=\bigl\{(1,3)\bigr\}\not\subseteq A$.
	
	It follows that there are $5=\frac{2\cdot 2+1}{2\cdot 3+1}\binom{3\cdot 3-2}{3-2}$ geometric multi-chains of $2$-filters of length~$2$.
\end{example}

Let us denote by $\NN_{n}^{(m)}$ the set of all geometric multi-chains of filters of length~$m$, and let us denote by $\NN_{n;t}^{(m)}$ the set of all geometric multi-chains of $t$-filters of length~$m$.  As noted before, we have $\NN_{n;t}^{(m)}\subseteq\NN_{n}^{(m)}$ with equality if and only if $t=1$.  

\begin{remark}
	One way to prove Conjecture~\ref{conj:geometric_chains_cardinality} is by exhibiting an explicit bijection from $\NC_{n;t}^{(m)}$ to $\NN_{n;t}^{(m)}$.
\end{remark}

Moreover, we denote by $\bigl(\NN_{n;t}^{(m)},\subseteq\bigr)$ the set $\NN_{n;t}^{(m)}$ ordered by componentwise inclusion.  

\begin{lemma}\label{lem:geometric_covers}
	Let $(V_{m},V_{m-1},\ldots,V_{1})$ and $(W_{m},W_{m-1},\ldots,W_{1})$ be two geometric multi-chains of $t$-filters that form a covering pair in $\bigl(\NN_{n;t}^{(m)},\subseteq\bigr)$.  Then, there exists a unique index $j\in[m]$ such that $V_{i}=W_{i}$ for $i\neq j$ and $V_{j}$ and $W_{j}$ differ in exactly one element.
\end{lemma}
\begin{proof}
	For $t=1$, the claim follows from \cite{athanasiadis05refinement}*{Theorem~3.6}, which states that the geometric multi-chains of filters correspond bijectively to regions in the fundamental chamber of the extended Shi arrangement.  Moreover, the adjacency graph of these regions corresponds to the poset diagram of $\bigl(\NN_{n}^{(m)},\subseteq\bigr)$.  Then, crossing a wall in the fundamental chamber corresponds to a covering pair in $\bigl(\NN_{n;t}^{(m)},\subseteq\bigr)$; the colour of this wall determines the index $j$ from the statement, and the normal vector of the hyperplane supporting this wall determines the unique element in which the $j$-th components of the multi-chains in question differ.
	
	For $t>1$, the claim follows from the fact that $T_{n;t}$ is a filter of $T_{n}$, which implies that covering pairs in $\bigl(\NN_{n;t}^{(m)},\subseteq\bigr)$ are covering pairs in $\bigl(\NN_{n}^{(m)},\subseteq\bigr)$.
\end{proof}

Now, let $\mathcal{V}=(V_{m},V_{m-1},\ldots,V_{1}),\mathcal{W}=(W_{m},W_{m-1},\ldots,W_{1})\in\NN_{n;t}^{(m)}$ be such that $\mathcal{W}$ covers $\mathcal{V}$ in $\bigl(\NN_{n;t}^{(m)},\subseteq\bigr)$.  We define 
\begin{equation}\label{eq:single_colored_floors}
	\fl_{m}(\mathcal{V},\mathcal{W})\defs W_{m}\setminus V_{m},
\end{equation}
which by Lemma~\ref{lem:geometric_covers} is either the empty set or a set consisting of a single pair $(i,j)$.  For $\mathcal{W}\in\NN_{n;t}^{(m)}$, we set
\begin{equation}\label{eq:colored_floors}
	\FL_{m}(\mathcal{W}) \defs \bigcup_{\substack{\mathcal{V}\in\NN_{n;t}^{(m)}\\
\mathcal{V}\;\text{is covered by}\;\mathcal{W}}}\fl_{m}(\mathcal{V},\mathcal{W}).
\end{equation}

Finally, we denote by 
\begin{displaymath}
	S_{n;t} \defs \bigl\{(t,t{+}1),(t{+}1,t{+}2),\ldots,(n{-}1,n)\bigr\}
\end{displaymath}
the set of minimal elements of $\bigl(T_{n;t},\preceq\bigr)$.

We may now use these definitions to conjecture a combinatorial explanation of the $H$-triangle.  For positive integers $m,n,t$ we define
\begin{equation}\label{eq:h_triangle_combin}
	\tilde{H}_{n;t}^{(m)}(x,y) \defs \sum_{\mathcal{V}\in\NN_{n;t}^{(m)}}x^{\lvert\FL(\mathcal{V})\rvert}y^{\lvert\FL(\mathcal{V})\cap S_{n;t}\rvert}.
\end{equation}

Note that, for $t=1$, we recover exactly the type-$A$ case of \cite{armstrong09generalized}*{Definition~5.3.1}.

\begin{conjecture}\label{conj:h_triangle}
	For positive integers $m,n,t$ we have $\tilde{H}_{n;t}^{(m)}(x,y)=H_{n;t}^{(m)}(x,y)$, \ie
	\begin{align*}
		\tilde{H}_{n;t}^{(m)}(x,y) & = \sum_{k=0}^{n-t}\sum_{h=0}^{n-t-k}\left(\binom{mn-t+1}{k}\binom{t+k+h-2}{h}\right.\\
		& \kern2cm\left.-m\binom{mn-t}{k-1}\binom{t+k+h-1}{h}\right)x^{n-t-k}y^{n-t-k-h}.
	\end{align*}
\end{conjecture}

\begin{example}\label{ex:geometric_2}
	Let us continue Example~\ref{ex:geometric_1}.  Figure~\ref{fig:geometric_232} shows $\bigl(\NN_{3;2}^{(2)},\subseteq\bigr)$, where the covering pairs are labelled by the map $\fl_{2}$ from \eqref{eq:single_colored_floors}.  With $S_{3;2}=\bigl\{(2,3)\bigr\}$, we compute 
	\begin{displaymath}
		H_{3;2}^{(2)}(x,y) = xy+x+3
	\end{displaymath}
	in agreement with Theorem~\ref{thm:fhm_correspondence}.
\end{example}

\begin{figure}
	\centering
	\begin{tikzpicture}\small
		\draw(1,1) node(n1){$\Bigl(\emptyset,\emptyset\Bigr)$};
		\draw(1,3) node(n2){$\Bigl(\emptyset,\bigl\{(1,3)\bigr\}\Bigr)$};
		\draw(1,5) node(n3){$\Bigl(\emptyset,\bigl\{(1,3),(2,3)\bigr\}\Bigr)$};
		\draw(1,7) node(n4){$\Bigl(\bigl\{(1,3)\bigr\},\bigl\{(1,3),(2,3)\bigr\}\Bigr)$};
		\draw(1,9) node(n5){$\Bigl(\bigl\{(1,3),(2,3)\bigr\},\bigl\{(1,3),(2,3)\bigr\}\Bigr)$};
		\draw(n1) -- (n2) node at (1,2) [text=white!50!black,fill=white,inner sep=2pt] {\tiny $\emptyset$};
		\draw(n2) -- (n3) node at (1,4) [text=white!50!black,fill=white,inner sep=2pt] {\tiny $\emptyset$};
		\draw(n3) -- (n4) node at (1,6) [text=white!50!black,fill=white,inner sep=2pt] {\tiny $\bigl\{(1,3)\bigr\}$};
		\draw(n4) -- (n5) node at (1,8) [text=white!50!black,fill=white,inner sep=2pt] {\tiny $\bigl\{(2,3)\bigr\}$};
	\end{tikzpicture}
	\caption{The poset $\bigl(\NN_{3;2}^{(2)},\subseteq\bigr)$.}
	\label{fig:geometric_232}
\end{figure}

\subsection{A Dyck path model for $m=1$}

Let us now focus on the case $m=1$.  In particular, we observe that now the ``geometric'' condition from Definition~\ref{def:geometric_multichain} is vacuous, and the set $\NN_{n;t}^{(1)}$ consists simply of the $t$-filters of $T_{n;t}$.  We will think of these filters equivalently as certain lattice paths.

A \defn{Dyck path} is a lattice path starting at the origin, ending on the $x$-axis, and which consists only of steps of the form $(1,1)$ (\defn{up-steps}) and $(1,-1)$ (\defn{down-steps}) while never going below the $x$-axis.  The \defn{length} of a Dyck path $P$ is its number of steps, denoted by $\ell(P)$.  A \defn{$t$-Dyck path} is a Dyck path that starts with $t$ up-steps.  Let $\Dyck_{n;t}$ denote the set of all $t$-Dyck paths of length $2n$.  

A \defn{valley} of a Dyck path is a coordinate on this path which is preceded by a down-step, and followed by an up-step.  

\begin{lemma}\label{lem:filters_paths}
	For positive integers $n,t$, the sets $\NN_{n;t}^{(1)}$ and $\Dyck_{n;t}$ are in bijection.
\end{lemma}
\begin{proof}
	Such a bijection, say $\Theta$, can be constructed by fitting a $t$-Dyck path below the $t$-filter such that the minimal elements of $\mathcal{V}\in\NN_{n;t}^{(1)}$ correspond to the valleys of $\Theta(\mathcal{V})\in\Dyck_{n;t}$.
\end{proof}

See Figure~\ref{fig:filters_ballot_paths} for an illustration.

\begin{figure}
	\centering
	\begin{tikzpicture}
		\def\x{.5};
		\foreach \i in {0,...,12}{
			\draw[gray!25!white](\i*\x,\i*\x) -- (2*\i*\x,0*\x) -- ({(12+\i)*\x},{(12-\i)*\x});
		}
		\begin{pgfonlayer}{background}
			\fill[gray!60!white](0*\x,0*\x) -- (5*\x,5*\x) -- (10*\x,0*\x) -- (11*\x,1*\x) -- (12*\x,0*\x) -- (13*\x,1*\x) -- (14*\x,0*\x) -- (15*\x,1*\x) -- (16*\x,0*\x) -- (17*\x,1*\x) -- (18*\x,0*\x) -- (19*\x,1*\x) -- (20*\x,0*\x) -- (21*\x,1*\x) -- (22*\x,0*\x) -- (23*\x,1*\x) -- (24*\x,0*\x) -- cycle;
		\end{pgfonlayer}
		\foreach \i in {5,...,11}{
			\foreach \j in {1,...,5}{
				\draw({(2*\i-\j+1)*\x},\j*\x) node[circle,draw,scale=.8](p\i\j){};
			}
		}
		\newcounter{k}
		\setcounter{k}{6}
		\foreach \i in {1,...,6}{
			\foreach \j in {1,...,\thek}{
				\draw({(4+\j+2*\i)*\x},{(5+\j)*\x}) node[circle,draw,scale=.8](q\i\j){};
			}
			\addtocounter{k}{-1}
		}
		\draw(p51) -- (p52) -- (p53) -- (p54) -- (p55);
		\draw(p61) -- (p62) -- (p63) -- (p64) -- (p65) -- (q11);
		\draw(p71) -- (p72) -- (p73) -- (p74) -- (p75) -- (q21) -- (q12);
		\draw(p81) -- (p82) -- (p83) -- (p84) -- (p85) -- (q31) -- (q22) -- (q13);
		\draw(p91) -- (p92) -- (p93) -- (p94) -- (p95) -- (q41) -- (q32) -- (q23) -- (q14);
		\draw(p101) -- (p102) -- (p103) -- (p104) -- (p105) -- (q51) -- (q42) -- (q33) -- (q24) -- (q15);
		\draw(p111) -- (p112) -- (p113) -- (p114) -- (p115) -- (q61) -- (q52) -- (q43) -- (q34) -- (q25) -- (q16);
		\draw(p55) -- (q11) -- (q12) -- (q13) -- (q14) -- (q15) -- (q16);
		\draw(p54) -- (p65) -- (q21) -- (q22) -- (q23) -- (q24) -- (q25);
		\draw(p53) -- (p64) -- (p75) -- (q31) -- (q32) -- (q33) -- (q34);
		\draw(p52) -- (p63) -- (p74) -- (p85) -- (q41) -- (q42) -- (q43);
		\draw(p51) -- (p62) -- (p73) -- (p84) -- (p95) -- (q51) -- (q52);
		\draw(p61) -- (p72) -- (p83) -- (p94) -- (p105) -- (q61);
		\draw(p71) -- (p82) -- (p93) -- (p104) -- (p115);
		\draw(p81) -- (p92) -- (p103) -- (p114);
		\draw(p91) -- (p102) -- (p113);
		\draw(p101) -- (p112);
		\draw(21*\x,2*\x) node[circle,draw,fill,scale=.8]{};
		\draw(20*\x,3*\x) node[circle,draw,fill,scale=.8]{};
		\draw(17*\x,4*\x) node[circle,draw,fill,scale=.8]{};
		\draw(19*\x,4*\x) node[circle,draw,fill,scale=.8]{};
		\draw(14*\x,5*\x) node[circle,draw,fill,scale=.8]{};
		\draw(16*\x,5*\x) node[circle,draw,fill,scale=.8]{};
		\draw(18*\x,5*\x) node[circle,draw,fill,scale=.8]{};
		\draw(13*\x,6*\x) node[circle,draw,fill,scale=.8]{};
		\draw(15*\x,6*\x) node[circle,draw,fill,scale=.8]{};
		\draw(17*\x,6*\x) node[circle,draw,fill,scale=.8]{};
		\draw(8*\x,7*\x) node[circle,draw,fill,scale=.8]{};
		\draw(12*\x,7*\x) node[circle,draw,fill,scale=.8]{};
		\draw(14*\x,7*\x) node[circle,draw,fill,scale=.8]{};
		\draw(16*\x,7*\x) node[circle,draw,fill,scale=.8]{};
		\draw(9*\x,8*\x) node[circle,draw,fill,scale=.8]{};
		\draw(11*\x,8*\x) node[circle,draw,fill,scale=.8]{};
		\draw(13*\x,8*\x) node[circle,draw,fill,scale=.8]{};
		\draw(15*\x,8*\x) node[circle,draw,fill,scale=.8]{};
		\draw(10*\x,9*\x) node[circle,draw,fill,scale=.8]{};
		\draw(12*\x,9*\x) node[circle,draw,fill,scale=.8]{};
		\draw(14*\x,9*\x) node[circle,draw,fill,scale=.8]{};
		\draw(11*\x,10*\x) node[circle,draw,fill,scale=.8]{};
		\draw(13*\x,10*\x) node[circle,draw,fill,scale=.8]{};
		\draw(12*\x,11*\x) node[circle,draw,fill,scale=.8]{};
		\draw[line width=2pt,green!50!gray](0*\x,0*\x) -- (7*\x,7*\x) -- (8*\x,6*\x) -- (10*\x,8*\x) -- (14*\x,4*\x) -- (15*\x,5*\x) -- (17*\x,3*\x) -- (18*\x,4*\x) -- (21*\x,1*\x) -- (22*\x,2*\x) -- (24*\x,0*\x);
	\end{tikzpicture}
	\caption{Illustration of Lemma~\ref{lem:filters_paths}.  The green path belongs to $\Dyck_{12;5}$, the black dots represent the corresponding element of $\NN_{12;5}^{(1)}$.}
	\label{fig:filters_ballot_paths}
\end{figure}

\begin{corollary}
	For $m=1$, Conjecture~\ref{conj:geometric_chains_cardinality} holds, \ie the cardinality of $\NN_{n;t}^{(1)}$ is $\frac{t+1}{n+1}\binom{2n-t}{n-t}$.
\end{corollary}
\begin{proof}
	By Lemma~\ref{lem:filters_paths} the cardinality of $\NN_{n;t}^{(1)}$ equals the cardinality of $\Dyck_{n;t}$, which by \cite{krattenthaler15lattice}*{Corollary~10.3.2} equals the desired quantity.
\end{proof}

Let us now describe what the map $\FL_{1}$ from \eqref{eq:colored_floors} counts, when composed with the bijection from Lemma~\ref{lem:filters_paths}.  

A \defn{peak} of a Dyck path is a coordinate on this path which is preceded by an up-step, and followed by a down-step.  We note that the set of peaks (or respectively, valleys) uniquely determines a Dyck path.  The \defn{height} of a peak (or respectively a valley) is its ordinate.  For $P\in\Dyck_{n;t}$, we write $\valley(P)$ for the number of valleys of $P$, we write $\peak(P)$ for the number of peaks of $P$, and we write $\return(P)$ for the number of valleys of $P$ of height~$0$.

For $P_{1},P_{2}\in\Dyck_{n;t}$ we say that $P_{1}$ \defn{dominates} $P_{2}$ if $P_{2}$ lies weakly below $P_{1}$ (as a lattice path).  We denote\footnote{Note that if a path dominates another, then it is smaller than the other path in this partial order.  Hence the subscript ``ddom'' for ``dual domination''.} this relation by $P_{1}\dom P_{2}$.  The partially ordered set $\bigl(\Dyck_{n;t},\dom\bigr)$ is isomorphic to $\bigl(\NN_{n;t}^{(1)},\subseteq\bigr)$ via the bijection from Lemma~\ref{lem:filters_paths}, and its dual was for instance studied in \cites{barcucci05distributive,bernardi09intervals,ferrari15chains,muehle15sb,muehle17heyting}.  Figure~\ref{fig:dom_42} shows $\bigl(\Dyck_{4;2},\dom\bigr)$.

\newcommand{\dycksetupA}[2]{
	\def\dx{#1};
	\def\ds{1.5*\dx};
	\foreach \i in {0,...,4}{
		\draw[gray!25!white](\i*\dx,\i*\dx) -- (2*\i*\dx,0*\dx) -- ({(4+\i)*\dx},{(4-\i)*\dx});
	}
	\begin{pgfonlayer}{background}
		\fill[white]((0*\dx,0*\dx) -- (4*\dx,4*\dx) -- (8*\dx,0*\dx) -- cycle;
		\fill[gray!60!white](0*\dx,0*\dx) -- (2*\dx,2*\dx) -- (4*\dx,0*\dx) -- (5*\dx,1*\dx) -- (6*\dx,0*\dx) -- (7*\dx,1*\dx) -- (8*\dx,0*\dx) -- cycle;
	\end{pgfonlayer}
	\coordinate(p23) at (4*\dx,1*\dx);
	\coordinate(p34) at (6*\dx,1*\dx);
	\coordinate(p13) at (3*\dx,2*\dx);
	\coordinate(p24) at (5*\dx,2*\dx);
	\coordinate(p14) at (4*\dx,3*\dx);
	
	\draw(p23) node[draw,circle,scale=\ds](v23){};
	\draw(p34) node[draw,circle,scale=\ds](v34){};
	\draw(p13) node[draw,circle,scale=\ds](v13){};
	\draw(p24) node[draw,circle,scale=\ds](v24){};
	\draw(p14) node[draw,circle,scale=\ds](v14){};
	\draw(v23) -- (v24);
	\draw(v23) -- (v13);
	\draw(v24) -- (v14);
	\draw(v13) -- (v14);
	\draw(v34) -- (v24);
	\draw(7*\dx,1.5*\dx) node[anchor=west,text=red!50!gray]{#2};
}

\begin{figure}
	\centering
	\begin{tikzpicture}
		\def\x{2.5};
		\def\s{.4};
		\coordinate(n1) at (2*\x,1*\x);
		\coordinate(n2) at (2*\x,2*\x);
		\coordinate(n3) at (1*\x,3*\x);
		\coordinate(n4) at (3*\x,3*\x);
		\coordinate(n5) at (2*\x,4*\x);
		\coordinate(n6) at (4*\x,4*\x);
		\coordinate(n7) at (1*\x,5*\x);
		\coordinate(n8) at (3*\x,5*\x);
		\coordinate(n9) at (2*\x,6*\x);
		\draw(n1) -- (n2) node[midway,text=white!50!black,fill=white,inner sep=2pt] {\tiny $\bigl\{(1,4)\bigr\}$};
		\draw(n2) -- (n3) node[midway,text=white!50!black,fill=white,inner sep=2pt] {\tiny $\bigl\{(1,3)\bigr\}$};
		\draw(n2) -- (n4) node[midway,text=white!50!black,fill=white,inner sep=2pt] {\tiny $\bigl\{(2,4)\bigr\}$};
		\draw(n3) -- (n5) node[midway,text=white!50!black,fill=white,inner sep=2pt] {\tiny $\bigl\{(2,4)\bigr\}$};
		\draw(n4) -- (n5) node[midway,text=white!50!black,fill=white,inner sep=2pt] {\tiny $\bigl\{(1,3)\bigr\}$};
		\draw(n4) -- (n6) node[midway,text=white!50!black,fill=white,inner sep=2pt] {\tiny $\bigl\{(3,4)\bigr\}$};
		\draw(n5) -- (n7) node[midway,text=white!50!black,fill=white,inner sep=2pt] {\tiny $\bigl\{(2,3)\bigr\}$};
		\draw(n5) -- (n8) node[midway,text=white!50!black,fill=white,inner sep=2pt] {\tiny $\bigl\{(3,4)\bigr\}$};
		\draw(n6) -- (n8) node[midway,text=white!50!black,fill=white,inner sep=2pt] {\tiny $\bigl\{(1,3)\bigr\}$};
		\draw(n7) -- (n9) node[midway,text=white!50!black,fill=white,inner sep=2pt] {\tiny $\bigl\{(3,4)\bigr\}$};
		\draw(n8) -- (n9) node[midway,text=white!50!black,fill=white,inner sep=2pt] {\tiny $\bigl\{(2,3)\bigr\}$};
		\draw(n1) node{
			\begin{tikzpicture}
				\dycksetupA{\s}{$1$}
				\draw[line width=2pt,green!50!gray](0*\dx,0*\dx) -- (4*\dx,4*\dx) -- (8*\dx,0*\dx);
			\end{tikzpicture}
		};
		\draw(n2) node{
			\begin{tikzpicture}
				\dycksetupA{\s}{$x$}
				\node[fill,circle,scale=\ds] at (p14){};
				\draw[line width=2pt,green!50!gray](0*\dx,0*\dx) -- (3*\dx,3*\dx) -- (4*\dx,2*\dx) -- (5*\dx,3*\dx) -- (8*\dx,0*\dx);
			\end{tikzpicture}
		};
		\draw(n3) node{
			\begin{tikzpicture}
				\dycksetupA{\s}{$x$}
				\node[fill,circle,scale=\ds] at (p13){};
				\node[fill,circle,scale=\ds] at (p14){};
				\draw[line width=2pt,green!50!gray](0*\dx,0*\dx) -- (2*\dx,2*\dx) -- (3*\dx,1*\dx) -- (5*\dx,3*\dx) -- (8*\dx,0*\dx);
			\end{tikzpicture}
		};
		\draw(n4) node{
			\begin{tikzpicture}
				\dycksetupA{\s}{$x$}
				\node[fill,circle,scale=\ds] at (p14){};
				\node[fill,circle,scale=\ds] at (p24){};
				\draw[line width=2pt,green!50!gray](0*\dx,0*\dx) -- (3*\dx,3*\dx) -- (5*\dx,1*\dx) -- (6*\dx,2*\dx) -- (8*\dx,0*\dx);
			\end{tikzpicture}
		};
		\draw(n5) node{
			\begin{tikzpicture}
				\dycksetupA{\s}{$x^{2}$}
				\node[fill,circle,scale=\ds] at (p13){};
				\node[fill,circle,scale=\ds] at (p14){};
				\node[fill,circle,scale=\ds] at (p24){};
				\draw[line width=2pt,green!50!gray](0*\dx,0*\dx) -- (2*\dx,2*\dx) -- (3*\dx,1*\dx) -- (4*\dx,2*\dx) -- (5*\dx,1*\dx) -- (6*\dx,2*\dx) -- (8*\dx,0*\dx);
			\end{tikzpicture}
		};
		\draw(n6) node{
			\begin{tikzpicture}
				\dycksetupA{\s}{$xy$}
				\node[fill,circle,scale=\ds] at (p14){};
				\node[fill,circle,scale=\ds] at (p24){};
				\node[fill,circle,scale=\ds] at (p34){};
				\draw[line width=2pt,green!50!gray](0*\dx,0*\dx) -- (3*\dx,3*\dx) -- (6*\dx,0*\dx) -- (7*\dx,1*\dx) -- (8*\dx,0*\dx);
			\end{tikzpicture}
		};
		\draw(n7) node{
			\begin{tikzpicture}
				\dycksetupA{\s}{$xy$}
				\node[fill,circle,scale=\ds] at (p13){};
				\node[fill,circle,scale=\ds] at (p14){};
				\node[fill,circle,scale=\ds] at (p23){};
				\node[fill,circle,scale=\ds] at (p24){};
				\draw[line width=2pt,green!50!gray](0*\dx,0*\dx) -- (2*\dx,2*\dx) -- (4*\dx,0*\dx) -- (6*\dx,2*\dx) -- (8*\dx,0*\dx);
			\end{tikzpicture}
		};
		\draw(n8) node{
			\begin{tikzpicture}
				\dycksetupA{\s}{$x^{2}y$}
				\node[fill,circle,scale=\ds] at (p13){};
				\node[fill,circle,scale=\ds] at (p14){};
				\node[fill,circle,scale=\ds] at (p24){};
				\node[fill,circle,scale=\ds] at (p34){};
				\draw[line width=2pt,green!50!gray](0*\dx,0*\dx) -- (2*\dx,2*\dx) -- (3*\dx,1*\dx) -- (4*\dx,2*\dx) -- (6*\dx,0*\dx) -- (7*\dx,1*\dx) -- (8*\dx,0*\dx);
			\end{tikzpicture}
		};
		\draw(n9) node{
			\begin{tikzpicture}
				\dycksetupA{\s}{$x^{2}y^{2}$}
				\node[fill,circle,scale=\ds] at (p13){};
				\node[fill,circle,scale=\ds] at (p14){};
				\node[fill,circle,scale=\ds] at (p23){};
				\node[fill,circle,scale=\ds] at (p24){};
				\node[fill,circle,scale=\ds] at (p34){};
				\draw[line width=2pt,green!50!gray](0*\dx,0*\dx) -- (2*\dx,2*\dx) -- (4*\dx,0*\dx) -- (5*\dx,1*\dx) -- (6*\dx,0*\dx) -- (7*\dx,1*\dx) -- (8*\dx,0*\dx);
			\end{tikzpicture}
		};
	\end{tikzpicture}
	\caption{The poset $\bigl(\Dyck_{4;2},\dom\bigr)$, overlaid by the poset $\bigl(\NN_{4;2}^{(1)},\subseteq\bigr)$.  The edges are labelled by the map $\fl_{1}$ from \eqref{eq:single_colored_floors}.  Next to each $2$-filter (respectively $2$-Dyck path), in dark red, is the term it contributes to $\tilde{H}_{4;2}^{(1)}(x,y)$.}
	\label{fig:dom_42}
\end{figure}

\begin{proposition}\label{prop:h_triangle_paths}
	For positive integers $n,t$ we have
	\begin{equation}\label{eq:h_triangle_paths}
		\tilde{H}_{n;t}^{(1)}(x,y) = \sum_{P\in\Dyck_{n;t}}x^{\valley(P)}y^{\return(P)}.
	\end{equation}
\end{proposition}
\begin{proof}
	Observe that, since we are in the case $m=1$, the empty set is not in the range of the map $\fl_{1}$.  Now let $\mathcal{W}\in\NN_{n;t}^{(1)}$ and $Q=\Theta(\mathcal{W})\in\Dyck_{n;t}$ be the $t$-Dyck path corresponding to $\mathcal{W}$ via the bijection from Lemma~\ref{lem:filters_paths}.  
	
	By construction, if $\mathcal{V}$ is covered by $\mathcal{W}$, then there exists a unique valley of $Q$, say at coordinate $(p,q)$, such that $\Theta(\mathcal{V})$ agrees with $Q$ except that it runs through $(p,q{+}2)$ instead of $(p,q)$.  More precisely, the path $\Theta(\mathcal{V})$ arises from $Q$ by changing the down-step before $(p,q)$ into an up-step, and by changing the up-step after $(p,q)$ into a down-step.  This establishes $\bigl\lvert\FL_{1}(\mathcal{W})\bigr\rvert=\valley(Q)$.
	
	Now, let $\mathcal{V}$ be an element in $\bigl(\NN_{n;t}^{(1)},\subseteq\bigr)$ covered by $\mathcal{W}$, and let $P=\Theta(\mathcal{V})$ be the corresponding $t$-Dyck path.  Again by construction, we find that $\fl_{1}(\mathcal{V},\mathcal{W})\in S_{n;t}$ if and only if $P$ has a peak $(p,2)$ and $Q$ has a valley $(p,0)$.  Therefore $\bigl\lvert\FL_{1}(\mathcal{W})\cap S_{n;t}\bigr\rvert=\return(Q)$.
\end{proof}

In Figure~\ref{fig:dom_42}, we have labelled the elements of the poset additionally by the term they contribute to $\tilde{H}_{4;2}^{(1)}$.  The readers may convince themselves that Proposition~\ref{prop:h_triangle_paths} holds.  We obtain 
\begin{displaymath}
	\tilde{H}_{4;2}^{(1)}(x,y) = x^{2}y^{2} + x^{2}y + x^{2} + 2xy + 3x + 1,
\end{displaymath}
which confirms Conjecture~\ref{conj:h_triangle} in this case.

\begin{remark}
	The combinatorial description of $\tilde{H}_{n;t}^{(1)}(x,y)$ described in Proposition~\ref{prop:h_triangle_paths} was found with the help of FindStat~\cite{findstat}.  In \cite{muehle19ballot}*{Section~5}, a combinatorial description of $H_{n;t}^{(1)}(x,y)$ is conjectured, which counts $t$-Dyck paths according to various kinds of peaks.  We leave it as a challenge for the readers to give a bijection on $\Dyck_{n;t}$ that exchanges the corresponding pairs of statistics.
\end{remark}

Before we conclude this article with the proof of Conjecture~\ref{conj:h_triangle} in the case $m=1$, we recall the Lagrange--B{\"u}rmann formula.

\begin{lemma}[\cite{henrici74applied}*{Theorem~1.9b}]\label{lem:lagrange_burmann}
	Let $f(z)$ and $g(z)$ be formal power series with $f(0)=0$, and let $F(z)$ be the compositional inverse of $f(z)$.  Then, for all non-zero integers $a$,
	\begin{displaymath}
		\coef{z^{a}}g\bigl(F(z)\bigr) = \frac{1}{a}\coef{z^{-1}}g'(z)f^{-a}(z).
	\end{displaymath}
\end{lemma}

\begin{theorem}\label{thm:h_triangle_m=1}
	For $m=1$, Conjecture~\ref{conj:h_triangle} holds, \ie $H_{n;t}^{(1)}(x,y)=\tilde{H}_{n;t}^{(1)}(x,y)$ for all positive integers $n$ and $t$.
\end{theorem}
\begin{proof}
	We first determine the generating function for (all) Dyck paths, where $z$ keeps track of the length and $x$ keeps track of the number of 
valleys, see \eqref{eq:D1} below. Based on that equation, we then determine the generating function for $t$-Dyck paths, where again $z$ keeps
track of the length and $x$ keeps track of the number of valleys, but where also $y$ keeps track of the valleys at height~$0$, see
\eqref{eq:D1t}.

\medskip
	A Dyck path can either be empty, it may have no valleys of height $0$, or it may have a valley of height $0$.  In the latter case, such a Dyck path may be decomposed as $uP_{1}dP_{2}$, where $u$ denotes an up-step, $d$ denotes a down-step, and $P_{1}$ and $P_{2}$ stand for Dyck paths, of which $P_{1}$ can possibly be empty.
	
	Let us abbreviate
	\begin{displaymath}
		\Dyck_{\bullet} \defs \bigcup_{n\geq 0}\Dyck_{n},
	\end{displaymath}
	and let
	\begin{displaymath}
		D_{1}(x;z) \defs \sum_{P\in\Dyck_{\bullet}}x^{\valley(P)}z^{\ell(P)/2}
	\end{displaymath}
	denote the generating function for Dyck paths by the number of valleys.  According to the reasoning in the first paragraph of this proof, we obtain the equation
	\begin{equation}\label{eq:D1}
		D_{1}(x;z) = 1 + zD_{1}(x;z) + xzD_{1}(x;z)\bigl(D_{1}(x;z)-1\bigr).
	\end{equation}
	Equivalently, writing $D(x;z)\defs D_{1}(x;z)-1$, this can be rewritten as
	\begin{equation}\label{eq:D1_inverse}
		\frac{D(x;z)}{\bigl(1+D(x;z)\bigr)\bigl(1+xD(x;z)\bigr)} = z.
	\end{equation}

\medskip
	In fact, we want to consider $t$-Dyck paths, which are precisely the Dyck paths that start with $t$ up-steps.  Let us write
	\begin{displaymath}
		\Dyck_{\bullet;t} \defs \bigcup_{n\geq 0}\Dyck_{n;t},
	\end{displaymath}
	and 
	\begin{displaymath}
		D_{1,t}(x,y;z) \defs \sum_{P\in\Dyck_{\bullet;t}}x^{\valley(P)}y^{\return(P)}z^{\ell(P)/2}.
	\end{displaymath}
	
	A $t$-Dyck path may be decomposed in the form
	\begin{equation}\label{eq:t_dyck_decomposition}
		\underbrace{u\ldots u}_{t\;\text{times}}(P_{1}d)\ldots(P_{t}d)(u\bar{P}_{1}d)\ldots(u\bar{P}_{s}d),
	\end{equation}
	for some non-negative integer $s$, where the $P_{i}$'s and $\bar{P}_{i}$'s again stand for (possibly empty) Dyck paths.  In $D_{1,t}(x,y;z)$, the path $P_{1}$ contributes a factor $zD_{1}(x;z)$, because it cannot be preceded by a valley.  If $i>1$, then the path $P_{i}$ is preceded by a valley (of height $>0$) if and only if it is not empty, and therefore contributes a factor $z\bigl(1+x\bigl(D_{1}(x;z)-1\bigr)\bigr)$.  The collection of the $\bar{P}_{i}$'s forms an ordinary Dyck path separated by its valleys of height $0$, and therefore, this piece contributes a factor
	\begin{displaymath}
		\sum_{s=0}^{\infty}\bigl(xyzD_{1}(x;z)\bigr)^{s} = \frac{1}{1-xyzD_{1}(x;z)}.
	\end{displaymath}
	We have just explained the following form for the generating function:
	\begin{equation} \label{eq:D1t}
		D_{1,t}(x,y;z) = \sum_{s=0}^{\infty}(xy)^{s}\bigl(zD_{1}(x;z)\bigr)^{s+1}\bigl(z\bigl(1+x\bigl(D_{1}(x;z)-1\bigr)\bigr)\bigr)^{t-1}.
	\end{equation}
		
	Proposition~\ref{prop:h_triangle_paths} then implies
	\begin{align*}
		\tilde{H}_{n;t}^{(1)}(x,y) & = \coef{z^{n}}D_{1,t}(x,y;z)\\
		& = \coef{z^{n}}\sum_{s=0}^{\infty}(xy)^{s}\bigl(zD_{1}(x;z)\bigr)^{s+1}\bigl(z\bigl(1+x\bigl(D_{1}(x;z)-1\bigr)\bigr)\bigr)^{t-1}\\
		& = \sum_{s=0}^{\infty}(xy)^{s}\coef{z^{n-t-s}}\bigl(1+D(x;z)\bigr)^{s+1}\bigl(1+xD(x;z)\bigr)^{t-1}.
	\end{align*}
	
	By \eqref{eq:D1_inverse}, the compositional inverse of $D(x;z)$ is $\frac{z}{(1+z)(1+xz)}$.  Lemma~\ref{lem:lagrange_burmann} with $F(z)=D(x;z)$ and $g(z)=(1+z)^{s+1}(1+xz)^{t-1}$ therefore yields
	\begin{align*}
	\tilde{H}_{n;t}^{(1)} (x,y) & = \sum_{s=0}^{\infty}(xy)^{s}\coef{z^{n-t-s}}g\bigl(D(x;z)\bigr)\\
		& = (xy)^{n-t}+\sum_{s=0}^{n-t-1}\frac{(xy)^{s}}{n-t-s}\coef{z^{-1}}\left(\vphantom{\frac{1}{k}}(s+1)(1+z)^{s}(1+xz)^{t-1}\right.\\
		& \kern2cm \left.\vphantom{\frac{1}{k}}+(t-1)(1+z)^{s+1}x(1+xz)^{t-2}\right)\left(\frac{z}{(1+z)(1+xz)}\right)^{-(n-t-s)}\\
		& = (xy)^{n-t}+\sum_{s=0}^{n-t-1}\frac{(xy)^{s}}{n-t-s}\coef{z^{n-t-s-1}}\left(\vphantom{\frac{1}{k}}(s+1)(1+z)^{n-t}(1+xz)^{n-s-1}\right.\\
		& \kern2cm \left.\vphantom{\frac{1}{k}}+(t-1)(1+z)^{n-t+1}x(1+xz)^{n-s-2}\right)\\
		& = (xy)^{n-t}+\sum_{s=0}^{n-t-1}\frac{(xy)^{s}}{n-t-s}\\
		&\kern1cm
		\cdot
		\coef{z^{n-t-s-1}}\left((s+1)\sum_{a=0}^{n-t}\sum_{b=0}^{n-s-1}\binom{n-t}{a}\binom{n-s-1}{b}x^{b}z^{a+b}\right.\\
		& \kern2cm \left.+(t-1)\sum_{a=0}^{n-t+1}\sum_{b=0}^{n-s-2}\binom{n-t+1}{a}\binom{n-s-2}{b}x^{b+1}z^{a+b}\right).
\end{align*}
At this point, we read off the coefficient of $z^{n-t-s-1}$ in the inner
expression, that is, the terms that correspond to $z^{a+b}=z^{n-t-s-1}$.
The underlying equation yields $a=n-t-s-1-b$. Thus, we obtain
\begin{align*}
	\tilde{H}_{n;t}^{(1)} (x,y) & = (xy)^{n-t}+\sum_{s=0}^{n-t-1}\frac{(xy)^{s}}{n-t-s}\\
	&\kern1cm \cdot \left((s+1)\sum_{b=0}^{n-s-1}\binom{n-t}{n-t-s-1-b}\binom{n-s-1}{b}x^{b}\right.\\
		& \kern2cm \left.+(t-1)\sum_{b=0}^{n-s-2}\binom{n-t+1}{n-t-s-1-b}\binom{n-s-2}{b}x^{b+1}\right).
\end{align*}
In the second sum, we replace the summation index $b$ by $b-1$, to get
\begin{align*}
	\tilde{H}_{n;t}^{(1)} (x,y) & = (xy)^{n-t}+\sum_{s=0}^{n-t-1}\frac{(xy)^{s}}{n-t-s}\\
	&\kern1cm\cdot
	\left((s+1)\sum_{b=0}^{n-s-1}\binom{n-t}{n-t-s-1-b}\binom{n-s-1}{b}x^{b}\right.\\
		& \kern2cm \left.+(t-1)\sum_{b=0}^{n-s-1}\binom{n-t+1}{n-t-s-b}\binom{n-s-2}{b-1}x^{b}\right)\\
		& = (xy)^{n-t}+\sum_{s=0}^{n-t-1}\sum_{b=0}^{n-s-1}\frac{x^{s+b}y^{s}}{n-t-s}\left((s+1)\binom{n-t}{n-t-s-1-b}\binom{n-s-1}{b}\right
		.\\
		& \kern2cm \left.+(t-1)\binom{n-t+1}{n-t-s-b}\binom{n-s-2}{b-1}\right).
\end{align*}
Next we put $b=r-s$. Then the above expression becomes
\begin{align*}
	\tilde{H}_{n;t}^{(1)} (x,y) & = (xy)^{n-t}+\sum_{s=0}^{n-t-1}\sum_{r=s}^{n-1}\frac{x^{r}y^{s}}{n-t-s}\left((s+1)\binom{n-t}{n-t-r-1}\binom{n-s-1}{r-s}\right
		.\\
		& \kern2cm \left.+(t-1)\binom{n-t+1}{n-t-r}\binom{n-s-2}{r-s-1}\right)\\
		& = (xy)^{n-t}+\sum_{s=0}^{n-t-1}\sum_{r=s}^{n-t}\frac{x^{r}y^{s}}{n-t-s}\left((s+1)\binom{n-t}{r+1}\binom{n-s-1}{r-s}\right
		.\\
		& \kern2cm \left.+(t-1)\binom{n-t+1}{r+1}\binom{n-s-2}{r-s-1}\right)\\
		& = (xy)^{n-t}+\sum_{s=0}^{n-t-1}\sum_{r=s}^{n-t}x^{r}y^{s}\frac{(n-t)!\,(n-s-2)!}{(r+1)!\,(n-r-1)!\,(r-s)!\,(n-t-r)!}\\
		& \kern2cm \cdot\left(\frac{(s+1)(n-s-1)(n-t-r)+(t-1)(n-t+1)(r-s)}{n-t-s}\right)\\
		& = (xy)^{n-t}+\sum_{s=0}^{n-t-1}\sum_{r=s}^{n-t}x^{r}y^{s}\frac{(n-t)!\,(n-s-2)!}{(r+1)!\,(n-r-1)!\,(r-s)!\,(n-t-r)!}\\
		& \kern2cm\cdot\left(\vphantom{\frac{1}{k}}ns-sr+rt-1+n-2r-st\right).
\end{align*}
The ``running" denominator $n-t-s$ has cancelled! 
We see by inspection that for $r=s=n-t$ the summand in the above double
sum equals $(xy)^{n-t}$, which is exactly the separate term in the
above expression. We may therefore integrate it into the double sum.
Consequently, we have
\begin{align*}
	\tilde{H}_{n;t}^{(1)} (x,y) & = \sum_{r=0}^{n-t}\sum_{s=0}^{r}x^{r}y^{s}\frac{(n-t)!\,(n-s-2)!}{(r+1)!\,(n-r-1)!\,(r-s)!\,(n-t-r)!}\\
		& \kern2cm\cdot\left(\vphantom{\frac{1}{k}}(n-t+1)(n-r-1)-(n-s-1)(n-t-r)\right)\\
		& = \sum_{r=0}^{n-t}\sum_{s=0}^{r}x^{r}y^{s}\left(\binom{n-t+1}{r+1}\binom{n-s-2}{r-s}-\binom{n-t}{r+1}\binom{n-s-1}{r-s}\right).
\end{align*}
In the final step, we replace $r$ by $n-t-r$ and $s$ by $n-t-r-s$.
Thereby, we arrive at
\begin{align*}
	\tilde{H}_{n;t}^{(1)} (x,y) & = \sum_{r=0}^{n-t}\sum_{s=0}^{n-t-r}x^{n-t-r}y^{n-t-r-s}\left(\binom{n-t+1}{r}\binom{t+r+s-2}{s}\right.\\
		& \kern2cm \left.-\binom{n-t}{r-1}\binom{t+r+s-1}{s}\right)\\
		& = H_{n;t}^{(1)}(x,y),
	\end{align*}
where the last line follows from the expression for $H_{n;t}^{(m)}(x,y)$
that we derived in \eqref{eq:H}.
\end{proof}

\begin{remark}
	By construction, it is obvious that multi-chains of $t$-filters correspond to multi-chains of $t$-Dyck paths with respect to $\dom$.  It remains to understand how Conditions~\eqref{eq:geometric_1} and \eqref{eq:geometric_2} translate to $t$-Dyck paths.
\end{remark}

\section{An extension to Coxeter groups}

In this last section, we outline a possible construction of $F$-, $H$- and $M$-triangles, for $m=1$, in the setting of parabolic quotients of finite Coxeter groups.  We wish to keep this part brief and explain this extension only in the case of the symmetric group.  The general definitions for finite Coxeter groups can be found in \cite{muehle19tamari}*{Section~6} or \cite{williams13cataland}*{Chapter~5}.  

We fix a composition $\alpha=(\alpha_{1},\alpha_{2},\ldots,\alpha_{r})$ of $n>0$ and colour a collection of nodes labelled by $1,2,\ldots,n$ according to the components of $\alpha$, that is, the nodes $1,2,\dots,\alpha_1$ are assigned the first colour,
the nodes $\alpha_1+1,\dots,\alpha_1+\alpha_2$ are assigned the second colour, etc.  An \defn{$\alpha$-partition} is a set partition of $[n]$ for which no block contains two nodes of the same colour.  The arc diagram of an $\alpha$-partition $\pi$ is obtained as follows: if $i$ and $j$ are adjacent members of some block of~$\pi$, then we connect the nodes labelled by $i$ and $j$ by an arc which leaves the node $i$ (which has colour $c_{i}$, say) to the bottom, passes below all nodes of colour $c_{i}$, passes above subsequent nodes, and finally enters the node labelled $j$ from above.  An $\alpha$-partition is \defn{non-crossing} if its diagram can be drawn such that no two arcs cross.  Let $\NC_{\alpha}$ denote the set of all non-crossing $\alpha$-partitions.  Figure~\ref{fig:alpha_partitions_22} shows the refinement order on $\NC_{(2,2)}$.  Clearly, if $\alpha=(t,1,1,\ldots,1)$ is a composition of $n$, then $\NC_{\alpha}=\NC_{n;t}^{(1)}$.

\newcommand{\partsetup}[1]{
	\def\dx{#1};
	\def\ds{5*\dx};
	\draw(.5*\dx,.5*\dx) node{};
	\draw(4.5*\dx,1.5*\dx) node{};

	\foreach \i in {1,...,4}{
		\coordinate(p\i) at (\i*\dx,1*\dx);
	}
	\begin{pgfonlayer}{background}
		\fill[blue!50!gray,opacity=.5] (.75*\dx,.75*\dx) -- (2.25*\dx,.75*\dx) -- (2.25*\dx,1.25*\dx) -- (.75*\dx,1.25*\dx) -- cycle;
		\fill[green!50!gray,opacity=.5]  (2.75*\dx,.75*\dx) -- (4.25*\dx,.75*\dx) -- (4.25*\dx,1.25*\dx) -- (2.75*\dx,1.25*\dx) -- cycle;
	\end{pgfonlayer}
	\draw(p1) node[fill,circle,scale=\ds](v1){};
	\draw(p2) node[fill,circle,scale=\ds](v2){};
	\draw(p3) node[fill,circle,scale=\ds](v3){};
	\draw(p4) node[fill,circle,scale=\ds](v4){};
}

\begin{figure}
	\begin{tikzpicture}
		\def\x{2.5};
		\def\y{1};
		\draw(2.5*\x,1*\y) node[inner sep=.5pt](n1){\begin{tikzpicture}\small
			\partsetup{.5}
		\end{tikzpicture}};
		\draw(1*\x,2*\y) node[inner sep=.5pt](n2){\begin{tikzpicture}\small
			\partsetup{.5}
			\draw[edge](v1) to[bend right=50] (2.5*\dx,1*\dx) to[bend left=50] (v4);
		\end{tikzpicture}};
		\draw(2*\x,2*\y) node[inner sep=.5pt](n3){\begin{tikzpicture}\small
			\partsetup{.5}
			\draw[edge](v1) to[bend right=50] (2.5*\dx,1*\dx) to[bend left=50] (v3);
		\end{tikzpicture}};
		\draw(3*\x,2*\y) node[inner sep=.5pt](n4){\begin{tikzpicture}\small
			\partsetup{.5}
			\draw[edge](v2) to[bend right=50] (2.5*\dx,1*\dx) to[bend left=50] (v4);
		\end{tikzpicture}};
		\draw(4*\x,2*\y) node[inner sep=.5pt](n5){\begin{tikzpicture}\small
			\partsetup{.5}
			\draw[edge](v2) to[bend right=50] (2.5*\dx,1*\dx) to[bend left=50] (v3);
		\end{tikzpicture}};
		\draw(2.5*\x,3*\y) node[inner sep=.5pt](n6){\begin{tikzpicture}\small
			\partsetup{.5}
			\draw[edge](v1) to[bend right=50] (2.6*\dx,1*\dx) to[bend left=50] (v3);
			\draw[edge](v2) to[bend right=50] (2.4*\dx,1*\dx) to[bend left=50] (v4);
		\end{tikzpicture}};
		\draw(n1) -- (n2);
		\draw(n1) -- (n3);
		\draw(n1) -- (n4);
		\draw(n1) -- (n5);
		\draw(n3) -- (n6);
		\draw(n4) -- (n6);
	\end{tikzpicture}
	\caption{The poset $\Bigl(\NC_{(2,2)},\dref\Bigr)$.}
	\label{fig:alpha_partitions_22}
\end{figure}

We may thus define the $M$-triangle of $\NC_{\alpha}$, denoted by $M_{\alpha}(x,y)$, analogously to Definition~\ref{def:m_triangle}.  By inspection of Figure~\ref{fig:alpha_partitions_22}, we obtain
\begin{displaymath}
	M_{(2,2)}(x,y) = x^{2}y^{2}-2xy^{2}+4xy+y^{2}-4y+1.
\end{displaymath}
Note that, by construction, the rank of $\Bigl(\NC_{n;t}^{(m)},\dref\Bigr)$ is $n-t$ which is precisely the exponent of the ``correction factor'' in the transformations of Theorem~\ref{thm:fhm_correspondence}.  The rank of $\Bigl(\NC_{(2,2)},\dref\Bigr)$ is $2$, and we obtain
\begin{align*}
	H_{(2,2)}(x,y) & = \bigl(x(y-1)+1\bigr)^{2}M_{(2,2)}\left(\frac{y}{y-1},\frac{x(y-1)}{x(y-1)+1}\right)\\
	& = x^{2}y^{2} + 2x^{2}y {\color{red!50!gray}- 2x^{2}} + 2xy + 2x + 1,\\
	F_{(2,2)}(x,y) & = y^{2}M_{(2,2)}\left(\frac{y+1}{y-x},\frac{y-x}{y}\right)\\
	& = x^{2} + 4xy + y^{2} + 2x + 4y + 1.
\end{align*}
Note that $H_{(2,2)}(x,y)$ has a negative coefficient, and can therefore not arise in the spirit of Proposition~\ref{prop:h_triangle_paths} as the sum over some lattice paths with respect to certain statistics.

Conversely, consider the Dyck path $P_{\alpha}\defs U^{\alpha_{1}}D^{\alpha_{1}}U^{\alpha_{2}}D^{\alpha_{2}}\cdots U^{\alpha_{r}}D^{\alpha_{r}}$ and denote by $\Dyck_{\alpha}$ the set of all Dyck paths of length $2n$ which stay weakly above $P_{\alpha}$.  It follows from \cite{muehle19tamari}*{Theorem~29} that $\NC_{\alpha}$ and $\Dyck_{\alpha}$ are in bijection, and, clearly, if $\alpha=(t,1,1,\ldots,1)$ is a composition of $n$, then $\Dyck_{\alpha}=\Dyck_{n;t}$.  Figure~\ref{fig:alpha_paths_212} lists the elements of $\Dyck_{(2,1,2)}$. 

\newcommand{\dycksetupB}[2]{
	\def\dx{#1};
	\def\ds{1.5*\dx};
	\foreach \i in {0,...,5}{
		\draw[gray!25!white](\i*\dx,\i*\dx) -- (2*\i*\dx,0*\dx) -- ({(5+\i)*\dx},{(5-\i)*\dx});
	}
	\begin{pgfonlayer}{background}
		\fill[gray!60!white](0*\dx,0*\dx) -- (2*\dx,2*\dx) -- (4*\dx,0*\dx) -- (5*\dx,1*\dx) -- (6*\dx,0*\dx) -- (8*\dx,2*\dx) -- (10*\dx,0*\dx) -- cycle;
	\end{pgfonlayer}
	\coordinate(p23) at (4*\dx,1*\dx);
	\coordinate(p34) at (6*\dx,1*\dx);
	\coordinate(p13) at (3*\dx,2*\dx);
	\coordinate(p24) at (5*\dx,2*\dx);
	\coordinate(p35) at (7*\dx,2*\dx);
	\coordinate(p14) at (4*\dx,3*\dx);
	\coordinate(p25) at (6*\dx,3*\dx);
	\coordinate(p15) at (5*\dx,4*\dx);
		
	\draw(p23) node(v23)[draw,circle,scale=\ds]{};
	\draw(p34) node(v34)[draw,circle,scale=\ds]{};
	\draw(p13) node(v13)[draw,circle,scale=\ds]{};
	\draw(p24) node(v24)[draw,circle,scale=\ds]{};
	\draw(p35) node(v35)[draw,circle,scale=\ds]{};
	\draw(p14) node(v14)[draw,circle,scale=\ds]{};
	\draw(p25) node(v25)[draw,circle,scale=\ds]{};
	\draw(p15) node(v15)[draw,circle,scale=\ds]{};
	\draw(v14) -- (v13) -- (v23) -- (v24) -- (v25) -- (v15) -- (v14) -- (v24) -- (v34) -- (v35) -- (v25);
	\draw(10.5*\dx,1.5*\dx) node[text=red!50!gray,scale=.8]{#2};
}

\begin{figure}
	\centering
	\begin{tikzpicture}\small
		\def\x{3};
		\def\y{2};
		\def\s{.25};
		\draw(1*\x,4*\y) node{\begin{tikzpicture}
			\dycksetupB{\s}{$x^{2}y^{2}$};
			\draw(p23) node[fill,circle,scale=\ds]{};
			\draw(p34) node[fill,circle,scale=\ds]{};
			\draw(p13) node[fill,circle,scale=\ds]{};
			\draw(p24) node[fill,circle,scale=\ds]{};
			\draw(p35) node[fill,circle,scale=\ds]{};
			\draw(p14) node[fill,circle,scale=\ds]{};
			\draw(p25) node[fill,circle,scale=\ds]{};
			\draw(p15) node[fill,circle,scale=\ds]{};
			\draw[line width=2pt,green!50!gray](0*\dx,0*\dx) -- (2*\dx,2*\dx) -- (4*\dx,0*\dx) -- (5*\dx,1*\dx) -- (6*\dx,0*\dx) -- (8*\dx,2*\dx) -- (10*\dx,0*\dx);
		\end{tikzpicture}};
		\draw(2*\x,4*\y) node{\begin{tikzpicture}
			\dycksetupB{\s}{$x^{2}y$};
			\draw(p23) node[fill,circle,scale=\ds]{};
			\draw(p13) node[fill,circle,scale=\ds]{};
			\draw(p24) node[fill,circle,scale=\ds]{};
			\draw(p35) node[fill,circle,scale=\ds]{};
			\draw(p14) node[fill,circle,scale=\ds]{};
			\draw(p25) node[fill,circle,scale=\ds]{};
			\draw(p15) node[fill,circle,scale=\ds]{};
			\draw[line width=2pt,green!50!gray](0*\dx,0*\dx) -- (2*\dx,2*\dx) -- (4*\dx,0*\dx) -- (6*\dx,2*\dx) -- (7*\dx,1*\dx) -- (8*\dx,2*\dx) -- (10*\dx,0*\dx);
		\end{tikzpicture}};
		\draw(3*\x,4*\y) node{\begin{tikzpicture}
			\dycksetupB{\s}{$xy$};
			\draw(p23) node[fill,circle,scale=\ds]{};
			\draw(p13) node[fill,circle,scale=\ds]{};
			\draw(p24) node[fill,circle,scale=\ds]{};
			\draw(p14) node[fill,circle,scale=\ds]{};
			\draw(p25) node[fill,circle,scale=\ds]{};
			\draw(p15) node[fill,circle,scale=\ds]{};
			\draw[line width=2pt,green!50!gray](0*\dx,0*\dx) -- (2*\dx,2*\dx) -- (4*\dx,0*\dx) -- (7*\dx,3*\dx) -- (10*\dx,0*\dx);
		\end{tikzpicture}};
		\draw(4*\x,4*\y) node{\begin{tikzpicture}
			\dycksetupB{\s}{$x^{2}y$};
			\draw(p34) node[fill,circle,scale=\ds]{};
			\draw(p13) node[fill,circle,scale=\ds]{};
			\draw(p24) node[fill,circle,scale=\ds]{};
			\draw(p35) node[fill,circle,scale=\ds]{};
			\draw(p14) node[fill,circle,scale=\ds]{};
			\draw(p25) node[fill,circle,scale=\ds]{};
			\draw(p15) node[fill,circle,scale=\ds]{};
			\draw[line width=2pt,green!50!gray](0*\dx,0*\dx) -- (2*\dx,2*\dx) -- (3*\dx,1*\dx) -- (4*\dx,2*\dx) -- (6*\dx,0*\dx) -- (8*\dx,2*\dx) -- (10*\dx,0*\dx);
		\end{tikzpicture}};
		\draw(1*\x,3*\y) node{\begin{tikzpicture}
			\dycksetupB{\s}{$x^{3}$};
			\draw(p13) node[fill,circle,scale=\ds]{};
			\draw(p24) node[fill,circle,scale=\ds]{};
			\draw(p35) node[fill,circle,scale=\ds]{};
			\draw(p14) node[fill,circle,scale=\ds]{};
			\draw(p25) node[fill,circle,scale=\ds]{};
			\draw(p15) node[fill,circle,scale=\ds]{};
			\draw[line width=2pt,green!50!gray](0*\dx,0*\dx) -- (2*\dx,2*\dx) -- (3*\dx,1*\dx) -- (4*\dx,2*\dx) -- (5*\dx,1*\dx) -- (6*\dx,2*\dx) -- (7*\dx,1*\dx) -- (8*\dx,2*\dx) -- (10*\dx,0*\dx);
		\end{tikzpicture}};
		\draw(2*\x,3*\y) node{\begin{tikzpicture}
			\dycksetupB{\s}{$x^{2}$};
			\draw(p13) node[fill,circle,scale=\ds]{};
			\draw(p24) node[fill,circle,scale=\ds]{};
			\draw(p14) node[fill,circle,scale=\ds]{};
			\draw(p25) node[fill,circle,scale=\ds]{};
			\draw(p15) node[fill,circle,scale=\ds]{};
			\draw[line width=2pt,green!50!gray](0*\dx,0*\dx) -- (2*\dx,2*\dx) -- (3*\dx,1*\dx) -- (4*\dx,2*\dx) -- (5*\dx,1*\dx) -- (7*\dx,3*\dx) -- (10*\dx,0*\dx);
		\end{tikzpicture}};
		\draw(3*\x,3*\y) node{\begin{tikzpicture}
			\dycksetupB{\s}{$x^{2}$};
			\draw(p13) node[fill,circle,scale=\ds]{};
			\draw(p35) node[fill,circle,scale=\ds]{};
			\draw(p14) node[fill,circle,scale=\ds]{};
			\draw(p25) node[fill,circle,scale=\ds]{};
			\draw(p15) node[fill,circle,scale=\ds]{};
			\draw[line width=2pt,green!50!gray](0*\dx,0*\dx) -- (2*\dx,2*\dx) -- (3*\dx,1*\dx) -- (5*\dx,3*\dx) -- (7*\dx,1*\dx) -- (8*\dx,2*\dx) -- (10*\dx,0*\dx);
		\end{tikzpicture}};
		\draw(4*\x,3*\y) node{\begin{tikzpicture}
			\dycksetupB{\s}{$x^{2}$};
			\draw(p13) node[fill,circle,scale=\ds]{};
			\draw(p14) node[fill,circle,scale=\ds]{};
			\draw(p25) node[fill,circle,scale=\ds]{};
			\draw(p15) node[fill,circle,scale=\ds]{};
			\draw[line width=2pt,green!50!gray](0*\dx,0*\dx) -- (2*\dx,2*\dx) -- (3*\dx,1*\dx) -- (5*\dx,3*\dx) -- (6*\dx,2*\dx) -- (7*\dx,3*\dx) -- (10*\dx,0*\dx);
		\end{tikzpicture}};
		\draw(1*\x,2*\y) node{\begin{tikzpicture}
			\dycksetupB{\s}{$x$};
			\draw(p13) node[fill,circle,scale=\ds]{};
			\draw(p14) node[fill,circle,scale=\ds]{};
			\draw(p15) node[fill,circle,scale=\ds]{};
			\draw[line width=2pt,green!50!gray](0*\dx,0*\dx) -- (2*\dx,2*\dx) -- (3*\dx,1*\dx) -- (6*\dx,4*\dx) -- (10*\dx,0*\dx);
		\end{tikzpicture}};
		\draw(2*\x,2*\y) node{\begin{tikzpicture}
			\dycksetupB{\s}{$xy$};
			\draw(p34) node[fill,circle,scale=\ds]{};
			\draw(p24) node[fill,circle,scale=\ds]{};
			\draw(p35) node[fill,circle,scale=\ds]{};
			\draw(p14) node[fill,circle,scale=\ds]{};
			\draw(p25) node[fill,circle,scale=\ds]{};
			\draw(p15) node[fill,circle,scale=\ds]{};
			\draw[line width=2pt,green!50!gray](0*\dx,0*\dx) -- (3*\dx,3*\dx) -- (6*\dx,0*\dx) -- (8*\dx,2*\dx) -- (10*\dx,0*\dx);
		\end{tikzpicture}};
		\draw(3*\x,2*\y) node{\begin{tikzpicture}
			\dycksetupB{\s}{$x^{2}$};
			\draw(p24) node[fill,circle,scale=\ds]{};
			\draw(p35) node[fill,circle,scale=\ds]{};
			\draw(p14) node[fill,circle,scale=\ds]{};
			\draw(p25) node[fill,circle,scale=\ds]{};
			\draw(p15) node[fill,circle,scale=\ds]{};
			\draw[line width=2pt,green!50!gray](0*\dx,0*\dx) -- (3*\dx,3*\dx) -- (5*\dx,1*\dx) -- (6*\dx,2*\dx) -- (7*\dx,1*\dx) -- (8*\dx,2*\dx) -- (10*\dx,0*\dx);
		\end{tikzpicture}};
		\draw(4*\x,2*\y) node{\begin{tikzpicture}
			\dycksetupB{\s}{$x$};
			\draw(p24) node[fill,circle,scale=\ds]{};
			\draw(p14) node[fill,circle,scale=\ds]{};
			\draw(p25) node[fill,circle,scale=\ds]{};
			\draw(p15) node[fill,circle,scale=\ds]{};
			\draw[line width=2pt,green!50!gray](0*\dx,0*\dx) -- (3*\dx,3*\dx) -- (5*\dx,1*\dx) -- (7*\dx,3*\dx) -- (10*\dx,0*\dx);
		\end{tikzpicture}};
		\draw(1*\x,1*\y) node{\begin{tikzpicture}
			\dycksetupB{\s}{$x^{2}$};
			\draw(p35) node[fill,circle,scale=\ds]{};
			\draw(p14) node[fill,circle,scale=\ds]{};
			\draw(p25) node[fill,circle,scale=\ds]{};
			\draw(p15) node[fill,circle,scale=\ds]{};
			\draw[line width=2pt,green!50!gray](0*\dx,0*\dx) -- (3*\dx,3*\dx) -- (4*\dx,2*\dx) -- (5*\dx,3*\dx) -- (7*\dx,1*\dx) -- (8*\dx,2*\dx) -- (10*\dx,0*\dx);
		\end{tikzpicture}};
		\draw(2*\x,1*\y) node{\begin{tikzpicture}
			\dycksetupB{\s}{$x^{2}$};
			\draw(p14) node[fill,circle,scale=\ds]{};
			\draw(p25) node[fill,circle,scale=\ds]{};
			\draw(p15) node[fill,circle,scale=\ds]{};
			\draw[line width=2pt,green!50!gray](0*\dx,0*\dx) -- (3*\dx,3*\dx) -- (4*\dx,2*\dx) -- (5*\dx,3*\dx) -- (6*\dx,2*\dx) -- (7*\dx,3*\dx) -- (10*\dx,0*\dx);
		\end{tikzpicture}};
		\draw(3*\x,1*\y) node{\begin{tikzpicture}
			\dycksetupB{\s}{$x$};
			\draw(p14) node[fill,circle,scale=\ds]{};
			\draw(p15) node[fill,circle,scale=\ds]{};
			\draw[line width=2pt,green!50!gray](0*\dx,0*\dx) -- (3*\dx,3*\dx) -- (4*\dx,2*\dx) -- (6*\dx,4*\dx) -- (10*\dx,0*\dx);
		\end{tikzpicture}};
		\draw(4*\x,1*\y) node{\begin{tikzpicture}
			\dycksetupB{\s}{$x$};
			\draw(p35) node[fill,circle,scale=\ds]{};
			\draw(p25) node[fill,circle,scale=\ds]{};
			\draw(p15) node[fill,circle,scale=\ds]{};
			\draw[line width=2pt,green!50!gray](0*\dx,0*\dx) -- (4*\dx,4*\dx) -- (7*\dx,1*\dx) -- (8*\dx,2*\dx) -- (10*\dx,0*\dx);
		\end{tikzpicture}};
		\draw(1*\x,0*\y) node{\begin{tikzpicture}
			\dycksetupB{\s}{$x$};
			\draw(p25) node[fill,circle,scale=\ds]{};
			\draw(p15) node[fill,circle,scale=\ds]{};
			\draw[line width=2pt,green!50!gray](0*\dx,0*\dx) -- (4*\dx,4*\dx) -- (6*\dx,2*\dx) -- (7*\dx,3*\dx) -- (10*\dx,0*\dx);
		\end{tikzpicture}};
		\draw(2*\x,0*\y) node{\begin{tikzpicture}
			\dycksetupB{\s}{$x$};
			\draw(p15) node[fill,circle,scale=\ds]{};
			\draw[line width=2pt,green!50!gray](0*\dx,0*\dx) -- (4*\dx,4*\dx) -- (5*\dx,3*\dx) -- (6*\dx,4*\dx) -- (10*\dx,0*\dx);
		\end{tikzpicture}};
		\draw(3*\x,0*\y) node{\begin{tikzpicture}
			\dycksetupB{\s}{$1$};
			\draw[line width=2pt,green!50!gray](0*\dx,0*\dx) -- (5*\dx,5*\dx) -- (10*\dx,0*\dx);
		\end{tikzpicture}};
	\end{tikzpicture}
	\caption{The nineteen elements of $\Dyck_{(2,1,2)}$.}
	\label{fig:alpha_paths_212}
\end{figure}

We may thus define the $\tilde{H}$-triangle of $\Dyck_{\alpha}$, denoted by $\tilde{H}_{\alpha}(x,y)$, analogously to Proposition~\ref{prop:h_triangle_paths}.  By inspection of Figure~\ref{fig:alpha_paths_212}, we obtain
\begin{displaymath}
	\tilde{H}_{(2,1,2)}(x,y) = x^{2}y^{2} + x^{3} + 2x^{2}y + 6x^{2} + 2xy + 6x + 1.
\end{displaymath}
Note that, by construction, the maximal number of valleys among the members of $\Dyck_{n;t}$ is $n-t$, which again is precisely the exponent of the ``correction factor'' in the transformations of Theorem~\ref{thm:fhm_correspondence}.  The maximal number of valleys among members of $\Dyck_{(2,1,2)}$ is $3$, and we obtain
\begin{align*}
	\tilde{M}_{(2,1,2)}(x,y) & = (1-y)^{3}\tilde{H}_{(2,1,2)}\left(\frac{y(x-1)}{1-y},\frac{x}{x-1}\right)\\
	& = x^{3}y^{3} - 12x^{2}y^{3} + 9x^{2}y^{2} + 25xy^{3} - 30xy^{2} - 14y^{3} + 8xy + 21y^{2} - 9y + 1,\\
	\tilde{F}_{(2,1,2)}(x,y) & = x^{3}\tilde{H}_{(2,1,2)}\left(\frac{x+1}{x},\frac{y+1}{x+1}\right)\\
	& = 14x^{3} + 4x^{2}y + xy^{2} + 25x^{2} + 4xy + 12x + 1.
\end{align*}
We observe that the polynomial $\tilde{M}_{(2,1,2)}(x,y)$ cannot arise from some graded poset $\mathbf{P}$ in the manner described in Definition~\ref{def:m_triangle}.  Indeed, if this were the case, then the constant term of $\tilde{M}_{(2,1,2)}(x,y)$ forces $\mathbf{P}$ to have a unique minimal element which must be covered by nine elements, because the coefficient of $y$ is $-9$.  However, the coefficient of $xy$ is $8$, implying that $\mathbf{P}$ has only eight elements of rank $1$ which is a contradiction.

We leave it as an exercise to the reader (perhaps using \cite{muehle21noncrossing}*{Figure~7}) to verify that 
\begin{displaymath}
	M_{(2,1,2)}(x,y) = x^{3}y^{3} - 4x^{2}y^{3} + 9x^{2}y^{2} + 5xy^{3} - 22xy^{2} - 2y^{3} + 8xy + 13y^{2} - 8y + 1.
\end{displaymath}
Using the transformations from Theorem~\ref{thm:fhm_correspondence}, once again with correction exponent $3$, we obtain
\begin{align*}
	H_{(2,1,2)}(x,y) & = \bigl(x(y-1)+1\bigr)^{3}M_{(2,1,2)}\left(\frac{y}{y-1},\frac{x(y-1)}{x(y-1)+1}\right)\\
	& = x^{3}y^{3} + x^{3}y^{2} + 3x^{3}y + 3x^{2}y^{2} {\color{red!50!gray}- 4x^{3}} + 6x^{2}y + 3xy + 5x + 1,\\
	F_{(2,1,2)}(x,y) & = y^{3}M_{(2,1,2)}\left(\frac{y+1}{y-x},\frac{y-x}{y}\right)\\
	& = 2x^{3} + 12x^{2}y + 4xy^{2} + y^{3} + 5x^{2} + 20xy + 4y^{2} + 4x + 8y + 1.
\end{align*}
Once again, $H_{(2,1,2)}(x,y)$ has a negative coefficient.

The previous examples show that the correspondence between $F$-, $H$- and $M$-triangles is not so well-behaved for arbitrary parabolic quotients of the symmetric group.  Conjectures~5.3~and~5.4 in \cite{muehle21noncrossing} claim that this correspondence holds precisely when $\alpha$ has at most one component larger than $1$.  Computer experiments suggest that this behaviour carries over similarly to Coxeter groups of other types.  Remarkably, the various $F$-triangles so obtained seem to always have non-negative coefficients, maybe hinting at interesting combinatorics to be discovered.


\end{document}